\documentclass[12pt,a4paper]{article}
\usepackage{amsmath}
\usepackage{amssymb}
\usepackage{accents}
\usepackage{hyperref}
\usepackage{mathtools}
\usepackage{amsthm}
\usepackage[usenames]{color}

\usepackage{graphicx}

\newcommand{\1}[1]{{\mathbf 1}{\{#1\}}}

\newcommand{\vr}{\varrho}

\newcommand{\eps}{\varepsilon}
\newcommand{\Z}{{\mathbb Z}}
\newcommand{\N}{{\mathbb N}}

\newcommand{\B}{{\mathsf B}}

\newcommand{\M}{{\mathcal M}}

\newcommand{\G}{{\mathcal G}}

\newcommand{\HH}{{\mathcal H}}
\newcommand{\R}{{\mathbb R}}

\newcommand{\ZZ}{{\mathcal Z}}

\newcommand{\T}{{\mathcal T}}
\newcommand{\NN}{{\mathcal N}}

\newcommand{\EE}{{\mathcal E}}
\newcommand{\F}{{\mathcal F}}

\newcommand{\gs}{{\mathfrak S}}

\let\phi=\varphi
\newcommand{\A}{{\mathcal A}}

\newcommand{\hN}{\widehat N}
\newcommand{\hB}{\widehat B}

\newcommand{\hG}{\widehat G}

\newcommand{\bga}{{\bar\gamma}}

\newcommand{\diam}{{\mathop{\mathrm{diam}}}}

\newcommand{\hk}{\widehat{k}}

\newcommand{\hL}{\widehat{L}}
\newcommand{\hR}{\widehat{R}}
\newcommand{\tL}{\widetilde{L}}

\newcommand{\IP}{{\mathbb P}}
\newcommand{\IE}{{\mathbb E}}

\DeclareMathSymbol{\widehatsym}{\mathord}{largesymbols}{"62}

\newcommand{\capa}{\mathop{\mathrm{cap}}}

\newcommand{\hm}{\mathop{\mathrm{hm}}\nolimits}

\newcommand{\Es}{\mathop{\mathrm{Es}}\nolimits}

\newcommand{\dist}{\mathop{\mathrm{dist}}}

\allowdisplaybreaks

\newtheorem{theo}{Theorem}[section]
\newtheorem{lem}[theo]{Lemma}

\newtheorem{prop}[theo]{Proposition}
\newtheorem{cor}[theo]{Corollary}
\newtheorem{rem}[theo]{Remark}

\title{Deviation probabilities for global
events in two dimensions (I): independent excursions
}
\author{Orph\'ee Collin$^{1}$ \and
 Serguei~Popov$^{2}$}

\begin{document}

\maketitle

{\footnotesize
\noindent $^{~1}$Technische Universit\"at Wien, Austria
\\
\noindent e-mail:
\texttt{orphee.collin@tuwien.ac.at}


\noindent $^{~2}$Centro de Matem\'atica, University of
Porto, Porto, Portugal\\
\noindent e-mail: \texttt{serguei.popov@fc.up.pt}

}

\begin{abstract}
We consider the field of local times generated by independent
excursions of the two-dimensional simple random walk between
the boundaries of two concentric disks, and study \emph{global events}:
informally, these are the events that have to occur everywhere
(i.e., if such an event occurs on a set, then it also occurs on all
its subsets). Examples include ``the set is completely covered'',
``the maximal local time on the set does not exceed a given level'',
or ``the set does not contain a covered disk of a given radius''.
For monotone global events, we show that the deviation probabilities
(i.e., the probabilities that the event occurs with atypically few
excursions, or does not occur with atypically many of them, or vice versa)
can be obtained from the typical behaviour on different scales only,
via a general decoupling argument based on soft local times.
\\[.3cm]\textbf{Keywords:} simple random walk,
global collections of events, soft local time,
relative harmonic measure
\\[.3cm]\textbf{AMS 2020 subject classifications:}
Primary 60G50; secondary 60F10, 60J10, 60J55
\\[.3cm]This work is licensed under a Creative Commons Attribution — 4.0 Interna-
tional licence (CC-BY 4.0, https://creativecommons.org/licenses/by/4.0/).
\end{abstract}

\section{Introduction and main results}
\label{s_intro}

Suppose that there is a field~$\ell=(\ell(x), x\in \Z^2)$ of \emph{local times}
on~$\Z^2$.
These are thought of as \emph{visit counts}
to the sites of~$\Z^2$ obtained in some natural way,
e.g.\ with a (discrete-time) random walk,
or independent excursions of a random walk.
We thus naturally assume that
$\ell(x)\in\Z_+:=\{0,1,2,\ldots\}$,
and we write $\ell\leq \ell'$ on $A\subset \Z^2$
if $\ell(x)\leq \ell'(x)$ for all $x\in A$.

We are going to study events that depend on the
local times field restricted to a finite set $A\subset \Z^2$;
i.e., such an event~$E^A$ is a subset of~$\Z_+^A$.
We call an event~$E^A$ \emph{increasing} if
$\ell\in E^A$ implies $\ell'\in E^A$ for all
$\ell'\geq \ell$, and \emph{decreasing} when
its complement is increasing; an event is called
\emph{monotone} if it is increasing or decreasing.

Next, assume that we have a collection of
events $\EE=(E^A, A\subset \Z^2\text{ finite})$. We call~$\EE$
a \emph{global collection of events}
(or simply a \emph{global event}, with some abuse of terminology) when, for all finite $A\subset \Z^2$,
$E^A$ implies~$E^{A'}$
for all~$A'\subset A$.
We always assume that~$\EE$ is translation invariant:
for all finite $A\subset\Z^2$ and $x\in\Z^2$,
$\ell\in E^{A+x}$ if and only if $\ell(\cdot+x)\in E^A$. Informally, a global
event is something that has to occur everywhere. As a natural example,
consider the events $E^A=\{\ell(x)\geq 1 \text{ for all }x\in A\}=\{A \text{ is covered}\}$:
note that, if a set is covered, then all of its subsets
are covered.
Other natural examples of global events are:
\begin{itemize}
 \item[(i)] minimal local time is at least~$u>0$;
 \item[(ii)] maximal local time does not exceed~$u>0$;
 \item[(iii)] in each disk of radius~$r$ which is fully
 inside~$A$, there is an uncovered site;\footnote{Formally,
if no such disk fits inside~$A$, we say that the event occurs.}
 this can be seen as a ``complement'' of an event
 ``there is a completely covered disk of radius~$r$
 somewhere''.
\end{itemize}
All these examples are monotone; (i) and ``$A$ is covered'' are increasing, and (ii)--(iii)
are decreasing. Of course, there are also global
events which are not monotone (for example,
``no two neighbouring sites of~$A$ have equal local times''),
but we do not consider them in this paper
since our methods only work for monotone events.

Let $\|\cdot\|$ be the Euclidean norm, and
let $\B(x,r)=\{y\in\Z^2 : \|x-y\|\leq r\}$
be the (discrete) disk of radius~$r$ centred at~$x$;
we abbreviate $\B(r):=\B(0,r)$.
We write $x\sim y$ when $x,y\in\Z^2$ are neighbours
(i.e., $\|x-y\|=1$), and denote by
$\partial A=\{x\in A: \text{there exists }y\notin A
\text{ such that }x\sim y\}$
the (internal) boundary of $A\subset\Z^2$.
Assume from now on
that~$\EE$ is a global collection of monotone events, and that, for a given~$n\geq 1$ (not necessarily integer), the local time field is obtained in the following
way: take~$m$ independent excursions between~$\B(n)$
and~$\partial\B(e n)$ with their starting points
chosen according to the \emph{relative}
harmonic measure~$\hm_{\B(n)}^{\B(en)}$
(see the precise definition in Section~\ref{s_exc_process} below).
Specifically, for $x\in\B(n)$, we let~$\ell(x)$ be the
total number of visits to~$x$ by all these~$m$ excursions
(i.e., the sum, over the~$m$ excursions, of the number of times
the corresponding trajectory is at~$x$, the initial site
of the excursion included), and we set $\ell(x)=0$ for all
$x\notin\B(n)$; that is, the visits to the sites
of $\B(en)\setminus\B(n)$ are not counted.
Let~$\IP_n^{(m)}[\cdot]$ be the associated
probability measure; we drop the lower index
when no confusion can occur.
For increasing
events, let us define
\begin{align}
 k_n &= \max\big\{j: \IP^{(j)}[E^{\B(n)}]\leq \tfrac{1}{7}\big\},
 \label{def_kn}
 \\
 k'_n &= \min\big\{j: \IP^{(j)}[E^{\B(n)}]\geq \tfrac{6}{7}\big\}.
 \label{def_hkn}
\end{align}
Here and in the sequel, $j$ ranges over~$\Z_+$, and we use the conventions
$\max\emptyset:=-\infty$, $\min\emptyset:=+\infty$, and $\max G:=+\infty$
for an unbounded $G\subset\Z_+$; so, the thresholds are always defined,
with values in $\Z_+\cup\{\pm\infty\}$.
Note that $\IP^{(j)}[E^{\B(n)}]$ is nondecreasing in~$j$
(adding an excursion can only increase the local times),
so $k_n<k'_n$ whenever $k'_n<\infty$; the specific values~$\frac{1}{7}$ and~$\frac{6}{7}$
are chosen for convenience (see the proofs in Section~\ref{s_proofs_main}).
One can say (for increasing events)
that the necessary excursion count to cause~$E^{\B(n)}$
is typically in $[k_n,k'_n]$.
Note that we do \emph{not} claim/require
the monotonicity of the above thresholds in~$n$.
Note also that these thresholds are defined for all~$n\geq 1$
(with the same collection~$\EE$); in the sequel, we will need
to compare them on different scales, i.e., for disks~$\B(n)$
of different radii (we refer to the radius~$n$ as the \emph{scale}).
These radii need not be integers (e.g., the thresholds~$k_{n/h}$
with real $h\geq 1$ appear in~\eqref{eq_regularity} below).
Analogously, for decreasing events
(then, $\IP^{(j)}[E^{\B(n)}]$ is nonincreasing in~$j$) we define
\begin{align}
 \underline k_n &= \max\big\{j: \IP^{(j)}[E^{\B(n)}]\geq \tfrac{6}{7}\big\},
 \label{def_kn_dec}
 \\
 \underline k'_n &= \min\big\{j: \IP^{(j)}[E^{\B(n)}]\leq \tfrac{1}{7}\big\}.
 \label{def_hkn_dec}
\end{align}
Again, $[\underline k_n,\underline k'_n]$
is the interval to which the \emph{maximal}
(for occurrence of $E^{\B(n)}$)
excursion count belongs.

Thus, for increasing events, $\IP^{(m)}[E^{\B(n)}]$ and
$\IP^{(m')}[(E^{\B(n)})^\complement]$
should typically be small when~$m$ is less than~$k_n$ and~$m'$ is greater than~$k'_n$;
likewise, for decreasing events, $\IP^{(m)}[E^{\B(n)}]$ and
$\IP^{(m')}[(E^{\B(n)})^\complement]$
should typically be small when~$m$ is greater than~$\underline k'_n$ and~$m'$ is less than~$\underline k_n$.
The main goal of this paper is to
obtain explicit estimates on these probabilities.
Here, we show that there is a fairly
general argument which permits us to obtain the deviation probabilities for global collections
of events only using the information about
their typical behaviour.
However, we first discuss some regularity assumptions on the quantities defined
in~\eqref{def_kn}--\eqref{def_hkn_dec};
these are not very restrictive (to the best of
the authors' knowledge, they are satisfied for
all \emph{interesting} applications of our results),
but simplify the exposition a lot.

\medskip
\noindent
\textbf{Condition~L.} Let $0<\kappa<M<\infty$
be some fixed constants.
In the case of increasing
events, we say that~$\EE$ satisfies Condition~L
at~$n$ if $k_n\geq \ln^\kappa n$
and $k'_n\leq \ln^M n$. In the case of decreasing
events, we say that~$\EE$ satisfies Condition~L
at~$n$ if $\underline k_n\geq \ln^\kappa n$
and $\underline k'_n\leq \ln^M n$.
Note that Condition~L at~$n$ implies that the corresponding thresholds
at~$n$ are finite (those on the other scales need not be).

\medskip

Next, let us abbreviate $\alpha:=0.1$
(this is the value of the parameter~$\tilde\alpha$
of~\eqref{admissible_collection} used in the proofs of our main results;
it controls the relative size of the smaller disks placed into a larger one).
Now we formulate our main result for the deviation
probabilities in the case when the event~$E^{\B(n)}$ occurs atypically, i.e., (for increasing events) the excursion count is ``lower than it should be'':
\begin{theo}
\label{t_exc_noncovering}
Fix $\eps_1\in(0,1)$ and $\eps_2>0$.
There exist positive constants $c_0,c_1,c_2,\eps_3$ and~$n_0$
(depending only on~$\kappa$, $M$, $\eps_1$, and~$\eps_2$)
such that the following holds for all $n\geq n_0$ and for
any global collection~$\EE$ of increasing events which satisfies
Condition~L at~$n$ (here and in Theorem~\ref{t_exc_covering},
$g_n$ is a nonnegative integer).
If $0\leq g_n\leq (1-\eps_1)k'_n$, then
\begin{equation}
\label{eq_exc_noncovering_lower}
\IP^{(k'_n-g_n)}[E^{\B(n)}] \geq
c_0\exp\Big(-c_1 \frac{g_n^2}{k'_n}\Big).
\end{equation}
If $0\leq g_n\leq (1-\eps_1)k_n$ and
\begin{equation}
\label{eq_regularity}
 k_n-\tfrac{1}{2}g_n\leq k_{n/h}
 \quad\text{for all }1\leq h\leq \Big(\frac{\alpha g_n}{\sqrt{k_n}}\Big)^{1+\eps_2},
\end{equation}
then
\begin{equation}
\label{eq_exc_noncovering_upper}
 \IP^{(k_n-g_n)}[E^{\B(n)}]
 \leq \exp\Big(-c_2 \frac{g_n^2}{k_n}\Big).
\end{equation}
In the case when~$\EE$ is a global collection
of decreasing events (which satisfies Condition~L at~$n$),
the estimate in~\eqref{eq_exc_noncovering_lower}
(with $\underline k_n$ in place of~$k'_n$ in its right-hand side)
is valid
for $\IP^{(\underline k_n+g_n)}[E^{\B(n)}]$ when
$0\leq g_n\leq \eps_3\underline k_n$,
and, if $0\leq g_n\leq (1-\eps_1)\underline k'_n$ and
\begin{equation}
\label{eq_regularity_dec}
 \underline k'_n+\tfrac{1}{2}g_n\geq \underline k'_{n/h}
 \quad\text{for all }1\leq h\leq \Big(\frac{\alpha g_n}{\sqrt{\underline k'_n}}\Big)^{1+\eps_2},
\end{equation}
the estimate in~\eqref{eq_exc_noncovering_upper}
(with $\underline k'_n$ in place of~$k_n$ in its right-hand side)
is valid for $\IP^{(\underline k'_n+g_n)}[E^{\B(n)}]$.
\end{theo}
The conditions~\eqref{eq_regularity}--\eqref{eq_regularity_dec} might look technical, but they are satisfied (and by a large margin) in typical examples, for all large enough~$n$. For instance, in covering problems one has $k_n = 2\ln^2 n - (1+o(1))\ln n\,\ln\ln n$ (or, more coarsely, $k_n = (c+o(1))\ln^2 n$), and in both cases it is straightforward to check that $k_n-\frac{1}{2}g_n \le k_{n/h}$ holds for all $h \le \exp(cg_n/\ln n)$, with a small enough constant $c>0$ (which is much more than what is required in~\eqref{eq_regularity}),
provided that
$g_n\in\big[c'\sqrt{k_n}\ln k_n,(1-\eps_1)k_n\big]$ in the former case and
$g_n\in\big[c'k_n,(1-\eps_1)k_n\big]$ in the latter one,
with an arbitrary fixed $c'>0$
(the lower bounds on~$g_n$ are needed because
the above asymptotics determine the differences $k_n-k_{n/h}$
only up to $o(\ln n\ln\ln n)$ (respectively, $o(\ln^2 n)$) terms,
which have to be dominated by~$g_n$).
Note also that, for $g_n<\alpha^{-1}\sqrt{k_n}$, the condition~\eqref{eq_regularity} is void.
These conditions are deliberately local in~$n$:
when the global event depends
on~$n$ through a threshold (e.g.\ $\{\min_{x\in A}\ell(x) \ge u\}$ with $u=c\ln^2 n$, see Section~\ref{s_LT_extrema}),
the collection~$\EE$ itself depends on~$n$, and all the thresholds (including~$k_{n/h}$ in~\eqref{eq_regularity}) are those of the collection that corresponds to~$n$; this is why, in our main results, the assumptions are only imposed at the given~$n$, and the constants (including~$n_0$) do not depend on~$\EE$: in this way, these results can be applied with a different collection for each~$n$.

Note that, in the decreasing case, the lower bound is only obtained
for $g_n\leq \eps_3\underline k_n$, i.e., for at most
$(1+\eps_3)\underline k_n$ excursions.
This restriction is due to the construction used in the proof
(there, the number of regeneration blocks cannot be smaller than
the number of excursions, see the end of the proof of the lower bound
in Section~\ref{s_proofs_main}); we do not claim that it is necessary.

Then, to formulate our result for the case when
the event~$E^{\B(n)}$ atypically fails to occur, we need a few definitions.
If the event $(E^{\B(n)})^\complement$
occurs, this means that $(E^{A})^\complement$
occurs for \emph{some} $A\subset \B(n)$;
in that case we say that~$A$ is a witness set
for $(E^{\B(n)})^\complement$ (or that $(E^{\B(n)})^\complement$ is witnessed on~$A$).
Let $\diam(A)=\max_{x,y\in A}\|x-y\|$ denote
the diameter of a finite set $A\subset \Z^2$.
So, on the event~$(E^{\B(n)})^\complement$, we define
the random variable
\[
 \vr_n = 1\vee \min_{\substack{A\subset\B(n):\\
  (E^{A})^\complement \text{ occurs}}}\diam(A)
\]
to be the smallest diameter of a set where
the event~$(E^{\B(n)})^\complement$ is witnessed
(to avoid some formal issues,
we set $\vr_n=1$ when the complementary event
is witnessed on a singleton).
Formally, we also set $\vr_n=+\infty$ on~$E^{\B(n)}$.
Then, we have
\begin{theo}
\label{t_exc_covering}
Fix $\eps_1\in(0,1)$.
There exist constants $C\in\R$, $c_1,c_2,c_3>0$, $\gamma_0>1$, and~$n_0$
(depending only on~$\kappa$, $M$, and~$\eps_1$),
and a (small enough) universal constant~$\tilde c>0$,
such that the following holds for all $n\geq n_0$ and for
any global collection~$\EE$ of monotone events which satisfies
Condition~L at~$n$.
\begin{itemize}
\item[(i)] If the events are increasing, then, for any integer $g_n\geq 0$,
\begin{equation}
\label{eq_exc_covering_lower}
 \IP^{(k_n+g_n)}\big[(E^{\B(n)})^\complement\big]
 \geq \sup_{1\leq r\leq n}
 \Bigg(
   \IP^{(k_n)}[\vr_n\leq r]
   \times \Big(1-\frac{1}{1\vee(\ln \frac{n}{r}+C)}\Big)^{g_n}
  \Bigg).
\end{equation}
If the events are decreasing and $g_n\leq (1-\eps_1)\underline{k}'_n$,
then
\begin{align}
\lefteqn{\IP^{(\underline{k}'_n-g_n)}
\big[(E^{\B(n)})^\complement\big]}
\nonumber\\
&\geq
\sup_{\substack{1\leq r\leq n,\\ \gamma\geq \gamma_0}}
\Bigg(
\Big(
\IP^{(\underline{k}'_n)}[\vr_n\leq r]
-c_1\frac{n^2}{r^2}
\exp\Big(-\frac{c_2\gamma\underline{k}'_n}{\ln\frac{en}{r}}\Big)
\Big)
\times
\Big(0\vee\Big(1-\frac{c_3\gamma}{\ln\frac{en}{r}}\Big)\Big)^{g_n}
\Bigg).
\label{eq_exc_covering_lower_decreasing}
\end{align}
\item[(ii)] If the events are increasing and
$g_n\leq (1-\eps_1)k'_n$, then
\begin{equation}
\label{eq_exc_covering}
 \IP^{(k'_n+g_n)}\big[(E^{\B(n)})^\complement\big]
 \leq m^{-2}
\end{equation}
for every $m\in\N$ such that
\begin{equation}
\label{df_m_n_covering}
m\leq \exp\Big(\tilde c\frac{g_n}{\sqrt{k'_n}}\Big)
\quad\text{and}\quad
k'_n+\tfrac{1}{2} g_n \geq k'_{\alpha^{-1}mn}.
\end{equation}
If the events are decreasing and $g_n\leq(1-\eps_1)\underline k_n$,
then~\eqref{eq_exc_covering}
is valid for $\IP^{(\underline k_n-g_n)}\big[(E^{\B(n)})^\complement\big]$,
now for every $m\in\N$ such that
\begin{equation}
\label{df_m_n_covering_dec}
m\leq \exp\Big(\tilde c\frac{g_n}{\sqrt{\underline k_n}}\Big)
\quad\text{and}\quad
\underline k_n-\tfrac{1}{2} g_n \leq \underline k_{\alpha^{-1}mn}.
\end{equation}
In both cases, we denote by~$m_n$ the maximal such~$m$
(with $m_n:=1$ if there is none), so that the corresponding
probability is at most~$m_n^{-2}$.
\end{itemize}
\end{theo}
Note that the second factor in
the right-hand side of~\eqref{eq_exc_covering_lower} equals
$\exp\big(-\frac{g_n}{\ln (n/r)}(1+o(1))\big)$
when $n/r\to\infty$.
Note also that decreasing~$\tilde c$ only makes the conditions
\eqref{df_m_n_covering}--\eqref{df_m_n_covering_dec} more restrictive,
so Theorem~\ref{t_exc_covering}~(ii) remains valid with any
smaller~$\tilde c$ (we use this observation in Section~\ref{s_appl}).

It is interesting to notice
the asymmetry between 
the two situations: for~$E^{\B(n)}$
to occur atypically, the local times have to be modified ``everywhere'',
and the cost of that is of Gaussian type, as in~\eqref{eq_exc_noncovering_lower}--\eqref{eq_exc_noncovering_upper};
on the other hand, for~$E^{\B(n)}$ not to occur
it is enough that the complementary event is witnessed on one
(possibly small) set, which is typically much cheaper
(compare Corollaries~\ref{c_CT_noncov} and~\ref{c_CT_cov},
or Corollaries~\ref{c_maxdisk_more} and~\ref{c_maxdisk_less} below).


At this point, let us discuss the differences between the independent excursions and the random walk on the torus. The latter case is left for a companion paper, but we can already note that (e.g., in the case of increasing events), while the typical behaviour and the upper tail deviations are essentially the same, the probability of lower deviations (as in Theorem~\ref{t_exc_noncovering}) is significantly smaller for the random walk on the torus.
The reason is that, for the random walk, if we force it to visit a certain set more often, then it must necessarily visit the complement of that set less often up to a fixed time; on the other hand, one can ``deviate'' the excursions from~$\B(n)$ at a considerably smaller cost.

To compare these two settings, recall the correspondence between
the time of the walk and the excursion counts. Consider the simple
random walk on the torus $\Z^2_n:=\Z^2/n\Z^2$, and let~$s$ be of order~$n$
and such that $s<\frac{n}{2e}$ (so that~$\B(es)$ can be seen as a subset
of~$\Z^2_n$). Then, the mean time between the starts of two consecutive
excursions between~$\B(s)$ and~$\partial\B(es)$ is
$\frac{2}{\pi}n^2\big(1+O(n^{-1})\big)$ (see Lemma~\ref{l_interexc_time}
and~\eqref{expr_esc_disks} below), so~$t$ steps of the walk correspond
to approximately $\frac{\pi t}{2n^2}$ such excursions; in particular,
$\frac{4}{\pi}\theta n^2\ln^2 n$ steps correspond to $2\theta\ln^2 n$
excursions.

This difference is already visible in the cover-time problem. For the random walk on a two-dimensional torus, the main result of~\cite{CGPV13} states that the cover time~$\T_n$ satisfies, for $\theta\in(0,1)$,
\[
  \IP\big[\T_n \le \tfrac{4}{\pi}\,\theta\,n^2\ln^2 n\big]
  = \exp\big(-n^{2(1-\sqrt{\theta})+o(1)}\big).
\]
This is of much smaller order than the corresponding deviation probability in the independent-excursions setting
(indeed, Theorem~\ref{t_exc_noncovering} implies
that, with $2\theta\ln^2 n$ independent excursions, the probability
that~$\B(n)$ is covered is at least
$\exp\big(-c(1-\theta)^2\ln^2 n\big)$, where $c>0$ depends on~$\theta$;
see Corollary~\ref{c_CT_noncov} below).
It is also worth noting that, in contrast to the case 
of excursions, finding sharp lower bounds for 
the probability of the covering event may be nontrivial
(see~\cite{LSX26} for the case of tori in dimensions $d\geq 3$).

Let us briefly comment on the related literature.
The typical behaviour of observables of the above kind is by now
rather well understood: the leading-order term of the cover time
of the two-dimensional torus was obtained in~\cite{DPRZ04},
and its second-order correction in~\cite{BK14}
(for the Brownian motion) and in~\cite{A21} (for the simple random walk),
see also~\cite{BRZ20} for the tightness of the cover time of
the two-dimensional sphere, and the recent paper~\cite{DRZ26},
where the convergence in distribution of the (square root of the,
suitably centred) Brownian cover time of the two-dimensional torus
to a randomly shifted Gumbel law is proved; the late points were studied
in~\cite{DPRZ06}, the size of the largest covered disk
in~\cite{DPR07}, and the extrema of local times in~\cite{A15}.
In these works, the analysis is based on the counts of excursions
between concentric circles on different scales.
Much less is known about the corresponding deviation probabilities:
besides the aforementioned paper~\cite{CGPV13}, let us mention
the exponential concentration bound for the cover times of general
graphs obtained in~\cite{Z18} (extending the results of~\cite{D14});
for the two-dimensional torus~$\Z^2_n$, it applies to the deviations
of the cover time of order at least $n^2\ln^{3/2}n$
(i.e., at least $\ln^{3/2}n$ in terms of the excursion counts).
Independent excursions naturally appear in the context of
two-dimensional random interlacements~\cite{CPV16,CP17}, and they
can be compared to the excursions of the random walk
by means of soft local times~\cite{SLT}; this is the reason why
we treat them first.

Let us state some conventions that are used throughout the paper.
We denote by $c,c',C,C',\ldots$ positive constants whose values
may change from line to line. The numbered constants
(such as $c_1$, $C_5$, etc.) keep their values within the statement
where they are introduced and in its proof; those of
Section~\ref{s_exc_SLT} are fixed for the rest of the paper.
Unless explicitly stated otherwise, all these constants,
as well as those implied in the $O(\cdot)$-notation,
do not depend on~$n$, on the excursion counts, and on
the collection~$\EE$: they may only depend on
the parameters $\kappa$, $M$, $\eps_1,\eps_2$ (and, in Section~\ref{s_appl},
on the parameters of the application at hand); constants that
do not depend on anything are called universal.
Likewise, ``for all large enough~$n$'' means ``for all $n\geq n_0$'',
where~$n_0$ only depends on the above parameters; in particular,
in the proofs of Theorems~\ref{t_exc_noncovering}
and~\ref{t_exc_covering}, the restrictions on~$n$ only use
the inequalities
$\ln^\kappa n\leq k_n<k'_n\leq \ln^M n$
(respectively,
$\ln^\kappa n\leq \underline k_n<\underline k'_n\leq \ln^M n$)
provided by Condition~L at~$n$.
Finally, whenever a number of excursions (or of disks) is not
an integer, its integer part is meant, unless indicated otherwise.

The rest of the paper is organized in the following way.
In Section~\ref{s_appl} we apply Theorems~\ref{t_exc_noncovering}
and~\ref{t_exc_covering} to the covering event, to the maximal
covered disk, and to the extrema of local times.
Section~\ref{s_prelims} contains the tools: the excursion process
and its stationary measure, the control of excursion counts,
and the soft local times. Theorems~\ref{t_exc_noncovering}
and~\ref{t_exc_covering} are proved in Section~\ref{s_proofs_main}.
Finally, in Section~\ref{s_brownian} we explain that these results
hold, \emph{mutatis mutandis}, for independent Brownian excursions
(with the occupation measure in place of the field of local times).
The Appendix contains a sketch of the proof of
Theorem~\ref{t_max_disk_typical} (on the typical size of
the maximal covered disk).

\section{Applications}
\label{s_appl}
In this section we discuss applications of our
results to some ``usual'' collections of global events.
In each case, the estimates on the deviation probabilities follow
from Theorems~\ref{t_exc_noncovering} and~\ref{t_exc_covering}
once the typical behaviour of the corresponding thresholds is known.
The latter is taken from the literature, where it was obtained
for the random walk on the torus; the adaptations of
these results to the setting of independent excursions are only sketched.

\subsection{The covering event}
\label{s_covering}
Our first example is $E^A = \{A \text{ is covered}\}$. As we noted before, this
is a collection of increasing events.
Then, it is known that both~$k_n$
and~$k'_n$ are equal to
$2\ln^2 n - (1+o(1))\ln n\ln\ln n$
(the main result of~\cite{A21} is about random
walk on a torus; to obtain the corresponding fact
about independent excursions, one compares them with the
excursions of that walk by means of
Lemma~2.11 of~\cite{CP17} and Corollary~\ref{c_SLT_domination},
and, for the lower threshold, one uses
Steps~1--3 of the proof of Proposition~2 of~\cite{CP24}, with
Corollary~\ref{c_SLT_domination} and Lemma~\ref{l_exc_count_one_pair}
in place of Lemmas~2.9--2.10 of~\cite{CP17}, to show that with high
probability there is an uncovered site in~$\B(n)$, and not just
somewhere on the torus).
Note that, with high probability, these comparisons change the excursion
counts by at most $O\big(\ln n\sqrt{\ln\ln n}\big)$; since this
is $o(\ln n\ln\ln n)$, the second-order term is not affected.
Abbreviate $\hk_n^{H}=2\ln^2 n - \ln n\ln\ln n
+ H\ln n$.

First, we discuss the deviation probabilities
in the noncovering regime (i.e., when, typically,
 $\B(n)$ should not be totally covered).
\begin{cor}
\label{c_CT_noncov}
 Assume that $H<0$ and that $\eps \ln\ln n \leq |H|
 \leq \gamma \ln n$, where $\eps>0$ and $\gamma\in(0,2)$.
 Then, there exist positive constants~$c_1,c_2$
 (depending only on~$\eps$ and~$\gamma$) such that,
 for all large enough~$n$,
\begin{equation}
\label{eq_CT_noncov}
 e^{-c_1 H^2} \leq \IP^{(\hk_n^{H})}\big[E^{\B(n)}\big]
 \leq e^{-c_2 H^2}.
\end{equation}
\end{cor}

\begin{proof}
This is an immediate consequence of Theorem~\ref{t_exc_noncovering}
(see also the discussion after Theorem~\ref{t_exc_noncovering} about the condition~\eqref{eq_regularity}).
\end{proof}

Note that, in the near-critical regime
$H=O(\ln\ln n)$, we obtain estimates of the
form $\exp\big(-c(\ln\ln n)^2\big)$.
In particular, this yields
Proposition~2 of~\cite{CP24}
(with~$e$ as the ratio of the radii of the two disks,
which is the case used there; the difference
between~$\hm_{\B(n)}$ and~$\hm_{\B(n)}^{\B(en)}$ is immaterial
by~\eqref{hm_vs_hmrel} below);
the one-scale argument used there gives the estimate
$\exp\big(-c\frac{(\ln\ln n)^2}{\ln\ln\ln n}\big)$
for that probability, cf.~\eqref{est_psi_n} below,
which suffices for the purposes of~\cite{CP24}.

Now, we consider the covering regime,
i.e., $\hk_n^{H}$ excursions with $H>0$.
\begin{cor}
\label{c_CT_cov}
 Assume that $H>0$ and that $\eps \ln\ln n \leq H
 \leq \gamma \ln n$, where $\eps>0$ and $\gamma\in(0,2)$.
 Then, there exist positive constants~$c_1,c_2$
 (depending only on~$\eps$ and~$\gamma$) such that,
 for all large enough~$n$,
\begin{equation}
\label{eq_CT_cov}
 e^{-c_1 H} \leq \IP^{(\hk_n^{H})}\big[(E^{\B(n)})^\complement\big]
 \leq e^{-c_2 H}.
\end{equation}
\end{cor}
\begin{proof}
First, note that we always have $\vr_n=1$
on the event $(E^{\B(n)})^\complement$
(if the event $E^{\B(n)}$ does not occur,
then there is an uncovered site~$z\in\B(n)$, so
$(E^{\{z\}})^\complement$ occurs).
 The lower bound is then a straightforward
 consequence of~\eqref{eq_exc_covering_lower}.

As for the upper bound, it is not difficult
to check that $m_n\geq e^{c H}$ with small enough~$c>0$, so we obtain the result from~\eqref{eq_exc_covering}.
\end{proof}
We note that the estimates in~\eqref{eq_CT_cov}
are not the sharpest ones that can be obtained;
at least in some regimes, one can have matching
constants there. We do not elaborate.

\subsection{Maximal covered disk}
\label{s_maxdisk}

The maximal size of a covered disk in~$A$
is defined in the following way:
\[
 \hR^A = \sup\big\{
  r>0: \text{ there exists $x$ such that $\B(x,r)\subset A$ and
  $\ell\geq 1$ on $\B(x,r)$}\big\}
\]
(with the convention $\sup\emptyset:=0$).

Our second example is:
$E^A = \big\{\hR^A\geq r\big\}^\complement=\big\{\hR^A< r\big\}$.
This is a
\emph{decreasing} global event.

As for the typical behaviour of the
maximal size of the covered disk,
first, recall from Section~\ref{s_covering} that,
with $2\ln^2 n$ excursions,
the whole~$\B(n)$ is likely to be covered, so
$\hR^{\B(n)}$ is likely to be
of order~$n$.
On the other hand, already with one excursion,
the size of the largest covered disk is roughly~$n^{1/2}$
due to the results of~\cite{DPR07}.
This \emph{suggests} that, with $2\theta \ln^2 n$ excursions where $\theta \in (0,1)$,
the typical value
of~$\hR^{\B(n)}$ should be
around~$n^{\chi(\theta)+o(1)}$, where $\chi(\cdot)$
is a strictly increasing function on~$[0,1]$
with $\chi(0)=\frac{1}{2}$ and $\chi(1)=1$.
It seems that the question of what is the actual value
of $\chi(\theta)$ was somehow overlooked
in the literature, even though it is possible
to show that $\chi(\theta)=\frac{1+\sqrt{\theta}}{2}$ for $\theta\in[0,1]$
using the multiscale excursion-count approach
of~\cite{A15,A21,BK14,DPRZ04}
(although this is not explicitly mentioned there).
More specifically, we have
\begin{theo}
\label{t_max_disk_typical}
For any $\theta\in(0,1)$ and $\eps>0$ we have
(with $\chi(\theta)=\frac{1+\sqrt{\theta}}{2}$)
\begin{equation}
\label{eq_max_disk_typical}
\IP^{(2\theta\ln^2 n)}\big[n^{\chi(\theta)-\eps}
\leq\hR^{\B(n)}
\leq n^{\chi(\theta)+\eps}\big]
 \to 1 \quad \text{ as }\quad n\to\infty.
\end{equation}
\end{theo}
Since the proof of this fact uses the tools developed
in Sections~\ref{s_prelims} and~\ref{s_proofs_main},
we postpone it to the Appendix.

Fix $a\in(1/2,1)$, and write\footnote{Note that,
if $a=\frac{1+\sqrt{\theta}}{2}$, then
$\theta=(2a-1)^2$.}
$K_n(a):=2(2a-1)^2\ln^2 n$.
For the decreasing global event
\[
E_n^{(a)}:=\{\hR^{\B(n)}<n^a\},
\]
let $\underline k_n(a)$ and $\underline k_n'(a)$ denote the lower and upper thresholds,
respectively (recall~\eqref{def_kn_dec}--\eqref{def_hkn_dec}). Theorem~\ref{t_max_disk_typical} implies that
\[
\underline k_n(a)=K_n(a)(1+o(1)),
\qquad
\underline k_n'(a)=K_n(a)(1+o(1)).
\]
Thus, for every fixed $\delta\in(0,1)$, the quantities
$
g_n^-:=\underline k_n'(a)-(1-\delta)K_n(a)
$
and
$
g_n^+:=(1+\delta)K_n(a)-\underline k_n(a)
$
are positive for all sufficiently large $n$, and
$g_n^-\sim\delta K_n(a)$, $g_n^+\sim\delta K_n(a)$.
Moreover, since the thresholds are nondecreasing in~$a$,
these asymptotics are uniform in~$a$ on compact subsets of~$(1/2,1)$;
in particular, when the event $\{\hR^A<n^a\}$ (with~$n$ and~$a$ kept fixed)
is considered on the disk of radius~$n^{1+x}$, the corresponding
thresholds are $2(2a-1-x)^2\ln^2 n\,(1+o(1))$,
uniformly in~$x$ in a small enough neighbourhood of~$0$.

First, we consider the case when there is
more than the critical number of excursions.
\begin{cor}
\label{c_maxdisk_more}
 Fix $\delta_1\in(0,1)$.
 For any $a\in(1/2,1)$ there exist positive constants~$c_a,C_a$
 (depending only on~$a$ and~$\delta_1$) such that
\begin{equation}
\label{eq_maxdisk_more}
 \exp\{-C_a \delta^2 K_n(a)\}
\le
\IP^{((1+\delta)K_n(a))}
\big[
\hR^{\B(n)}<n^a
\big]
\le
\exp\{-c_a \delta^2 K_n(a)\}
\end{equation}
for all sufficiently large~$n$ (depending on~$\delta$),
where the upper bound holds for every fixed $\delta\in(0,\delta_1]$,
and the lower bound holds for every fixed $\delta\in(0,\delta_0)$,
with some $\delta_0>0$ not depending on~$a$.
\end{cor}

\begin{proof}
We apply Theorem~\ref{t_exc_noncovering}
to the decreasing global event~$E_n^{(a)}$.
The lower bound in~\eqref{eq_maxdisk_more}
immediately follows from~\eqref{eq_exc_noncovering_lower}:
indeed, $\IP^{((1+\delta)K_n(a))}=\IP^{(\underline k_n(a)+g_n^+)}$
and $g_n^+\sim\delta\,\underline k_n(a)$, so, with
$\delta_0:=\eps_3$ (where~$\eps_3$ is from
Theorem~\ref{t_exc_noncovering}), we have
$g_n^+\leq\eps_3\underline k_n(a)$ for all large enough~$n$
when $\delta<\delta_0$.
As for the upper bound,
it is also straightforward: we use~\eqref{eq_exc_noncovering_upper}
with $\eps_1:=\frac{1-\delta_1}{2}$ (again, see the
discussion after Theorem~\ref{t_exc_noncovering}).
\end{proof}

In the case of fewer than the critical number of excursions, we obtain a polynomial decay
of deviation probabilities:
\begin{cor}
\label{c_maxdisk_less}
 Fix $\delta_1\in(0,1)$.
 For any $a\in(1/2,1)$ there exist positive constants~$c_a,C_a$
 (depending only on~$a$ and~$\delta_1$) such that,
 for every fixed $\delta\in(0,\delta_1]$ and all sufficiently large~$n$
 (depending on~$\delta$),
\begin{equation}
\label{eq_maxdisk_less}
n^{-C_a \delta}
\le
\IP^{((1-\delta)K_n(a))}
\big[
\hR^{\B(n)}\geq n^a
\big]
\le
n^{-c_a \delta}.
\end{equation}
\end{cor}
\begin{proof}
For the upper bound, recall that, for the disk of
radius~$n^{1+x}$ (and the same event), the thresholds
are asymptotic to $2(2a-1-x)^2\ln^2 n$; an elementary
calculation then implies that
(if~$\tilde c$ is small enough)
$m_n = n^{\sqrt{2}\tilde c(2a-1)\delta(1+o(1))}$,
which leads to
\[
  \IP^{((1-\delta)K_n(a))}
\big[\hR^{\B(n)}\geq n^a\big]
\leq  n^{-2\sqrt{2}\tilde c(2a-1)\delta(1+o(1))}.
\]

We now prove the lower bound in~\eqref{eq_maxdisk_less}.
Recall that
$g_n^-=\underline k_n'(a)-(1-\delta)K_n(a)=\delta K_n(a)(1+o(1))$.
On the event
$\{\hR^{\B(n)}\geq n^a\}$, there exists a covered disk
$D=\B(x,r')\subset\B(n)$ with $r'\leq n^a$ and $\hR^{D}\geq n^a$
(the supremum in the definition of~$\hR^{\B(n)}$ need not be attained,
so one may have to take $r'<n^a$), and hence a witness set
of diameter at most~$2n^a$.
We therefore apply the decreasing-event lower bound in
Theorem~\ref{t_exc_covering} with $g_n=g_n^-$ and $r=2n^a$
(and with $\eps_1:=\frac{1-\delta_1}{2}$, so that
$g_n^-\leq(1-\eps_1)\underline k'_n(a)$ for all large enough~$n$).

Choose $\gamma>0$ sufficiently large but fixed.
It is then straightforward to obtain
that (with large enough~$\gamma$)
the first factor in the right-hand side
of~\eqref{eq_exc_covering_lower_decreasing}
is at least~$1/2$; the second factor is clearly
at least $n^{-C\delta}$ for some~$C$
(depending on~$a$ and~$\delta_1$).
This proves the lower bound in~\eqref{eq_maxdisk_less}.
\end{proof}

\subsection{Extrema of local times}
\label{s_LT_extrema}
As one more illustration, consider the minimal and the maximal
local times on~$A$:
\[
\tL^A=\min_{y\in A}\ell(y), \qquad \hL^A=\max_{y\in A}\ell(y).
\]
For any $u>0$, $\big(\{\hL^A\leq u\}, A\subset\Z^2\text{ finite}\big)$
is a global collection of decreasing events, and
$\big(\{\tL^A\geq u\}, A\subset\Z^2\text{ finite}\big)$
is a global collection of increasing events; in both cases,
the complementary event is witnessed on a singleton, so
$\vr_n=1$ whenever it occurs. Denote
\[
\lambda^+(\theta)=\frac{4}{\pi}(1+\sqrt{\theta})^2,
\qquad
\lambda^-(\theta)=\frac{4}{\pi}(\sqrt{\theta}-1)^2.
\]
It can be obtained from the results of~\cite{A15} and their proofs
(the arguments there can be restricted to the sites of a fixed disk,
and the passage from the continuous-time walk on the torus to
independent excursions is as in Section~\ref{s_covering})
that, for any fixed $\eps>0$,
as $n\to\infty$,
\begin{equation}
\label{eq_extrema_typical}
\begin{aligned}
\IP^{(2\theta\ln^2 n)}\Big[\Big|\frac{\hL^{\B(n)}}{\ln^2 n}-\lambda^+(\theta)\Big|
\leq \eps\Big] &\to 1
\quad \text{for all }\theta>0,\\
\IP^{(2\theta\ln^2 n)}\Big[\Big|\frac{\tL^{\B(n)}}{\ln^2 n}-\lambda^-(\theta)\Big|
\leq \eps\Big] &\to 1
\quad \text{for all }\theta>1
\end{aligned}
\end{equation}
(when $\theta<1$, the disk~$\B(n)$ is typically not covered,
so $\tL^{\B(n)}=0$).
Hence, with $u=\lambda^+(\theta)\ln^2 n$ (respectively, $u=\lambda^-(\theta)\ln^2 n$),
all the corresponding thresholds are asymptotic to~$2\theta\ln^2 n$
(uniformly in~$\theta$ on compact sets, as in Section~\ref{s_maxdisk}),
and we obtain the following
\begin{cor}
\label{c_extrema}
Let~$\delta_0$ be as in Corollary~\ref{c_maxdisk_more},
and fix $\delta\in(0,\delta_0)$.
Then, there exist positive constants $c,C$ (depending only on~$\theta$)
such that, for all large enough~$n$ (depending on~$\delta$),
\begin{align}
e^{-C\delta^2\ln^2 n}
&\leq
\IP^{(2(1+\delta)\theta\ln^2 n)}
\big[\hL^{\B(n)}\leq \lambda^+(\theta)\ln^2 n\big]
\leq
e^{-c\delta^2\ln^2 n},
\label{eq_max_extrema_deviation}
\\
n^{-C\delta}
&\leq
\IP^{(2(1-\delta)\theta\ln^2 n)}
\big[\hL^{\B(n)}>\lambda^+(\theta)\ln^2 n\big]
\leq
 n^{-c\delta}
\label{eq_max_extrema_covering}
\end{align}
for any fixed $\theta>0$, and
\begin{align}
e^{-C\delta^2\ln^2 n}
&\leq
\IP^{(2(1-\delta)\theta\ln^2 n)}
\big[\tL^{\B(n)}\geq \lambda^-(\theta)\ln^2 n\big]
\leq
e^{-c\delta^2\ln^2 n},
\label{eq_min_extrema_deviation}
\\
n^{-C\delta}
&\leq
\IP^{(2(1+\delta)\theta\ln^2 n)}
\big[\tL^{\B(n)}<\lambda^-(\theta)\ln^2 n\big]
\leq
 n^{-c\delta}
\label{eq_min_extrema_covering}
\end{align}
for any fixed $\theta>1$.
\end{cor}

\begin{proof}
The estimates~\eqref{eq_max_extrema_deviation}
and~\eqref{eq_min_extrema_deviation} are straightforward consequences
of Theorem~\ref{t_exc_noncovering}, in the same way as in the proof
of Corollary~\ref{c_maxdisk_more}.
The lower bounds in~\eqref{eq_max_extrema_covering}
and~\eqref{eq_min_extrema_covering} follow from
Theorem~\ref{t_exc_covering}~(i) with $r=1$
(and, in~\eqref{eq_exc_covering_lower_decreasing}, with
a large enough~$\gamma$ depending on~$\theta$).
As for the upper bounds there, note that, by~\eqref{eq_extrema_typical},
for the disk of radius~$n^{1+x}$ (and the same~$u$)
the thresholds are asymptotic to $2(\sqrt{\theta}-x)^2\ln^2 n$
for the maximal local time, and to $2(\sqrt{\theta}+x)^2\ln^2 n$
for the minimal one; it is then elementary to obtain that
(if~$\tilde c$ is small enough)
$m_n=n^{\sqrt{2}\tilde c\delta\sqrt{\theta}(1+o(1))}$ in both cases,
and we use Theorem~\ref{t_exc_covering}~(ii).
\end{proof}

\section{Preliminaries: more notation and some tools}
\label{s_prelims}

\subsection{Excursion process and its stationary measure}
\label{s_exc_process}
Let $(S_k, k\geq 0)$ be the simple random walk (SRW) on~$\Z^2$
or on the torus~$\Z^2_n=\Z^2/n\Z^2$ (where the neighbourhood relation
and the boundaries of sets are defined in the same way as on~$\Z^2$).
We write~$\IP_x$ for the law of the walk started at~$x$,
and~$\IP_\mu$ when the starting point is chosen according
to a probability measure~$\mu$. For a set~$A$, let
\[
 \tau_A = \min\{k\geq 0: S_k\in A\},
 \qquad
 \tau_1(A) = \min\{k\geq 1: S_k\in A\}
\]
be the entrance and the hitting times of~$A$.
Also, for a finite $A\subset\Z^2$ we denote by
$\hm_A(x)=\lim_{\|y\|\to\infty}\IP_y[S_{\tau_A}=x]$, $x\in\partial A$,
the harmonic measure on~$A$.

We reproduce here
the construction from Section~3.4 of~\cite{CPV16},
with some mild changes.
Assume that two sets~$A$ and~$A'$ (which are subsets
of~$\Z^2_n$ or of $\Z^2$) are such that
$A\subset (A'\setminus\partial A')$.
By definition, an excursion~$Z$ from~$A$ to~$\partial A'$
is a simple random walk path
$Z=(Z_0, Z_1,\ldots, Z_m)$,
where $Z_0\in\partial A$, $Z_m\in\partial A'$,
$Z_k\notin\partial A'$ and $Z_k\sim Z_{k+1}$ for $k<m$. With some abuse
of notation, we denote by $Z_{\text{st}}:=Z_0$ and
$Z_{\text{end}}:=Z_m$ the starting and the ending points
of the excursion.
 From now on, when working with SRW on~$\Z^2_n$,
we will assume (unless explicitly stated otherwise)
that we are working with the \emph{stationary doubly
infinite version of the process}, i.e.,
we choose $S_0\in \Z^2_n$ uniformly at random and
let $(S_k, k\geq 0)$ and $(S_{-k}, k\geq 0)$
be conditionally (on the value of~$S_0$) independent
SRWs.
 To define the excursions, consider
the following sequence of stopping times:
\begin{align*}
 J_0 &= \tau_{\partial A},\\
 D_0 &= \inf\{t> J_0 : S_t \in \partial A'\},
\end{align*}
and
\begin{align*}
 J_k &= \inf\{t> D_{k-1} : S_t \in \partial A\},\\
 D_k &= \inf\{t> J_k : S_t \in \partial A'\},
\end{align*}
for $k\geq 1$. Then,
 $Z^{(i)}=(S_{J_i}, \ldots, S_{D_i})$ is the $i$th excursion
of~$S$ between~$\partial A$ and~$\partial A'$, for $i\geq 0$.
We also denote by
\begin{equation}
\label{df_E0}
\Gamma_0 = \big\{S_0\in\partial A, \text{there exists $m>0$
 such that }S_{-m}\in\partial A',
 S_{-k}\notin A\text{ for }0<k<m\big\}
\end{equation}
the event that the $0$th excursion starts at time $0$
in the doubly-infinite version of the process.

In the following, for the sake of conciseness,
we refer to excursions from~$\B(r)$
to~$\partial\B(R)$ (where $r+1<R$) as
``$r\to R$ excursions''.

Let us denote
\[
 \Es_A^{A'}
 = \sum_{z\in\partial A}\IP_z[\tau_1(\partial A')
<\tau_1(A)].
\]
Then, define the following measure on~$\partial A$:
\begin{align}
\hm_A^{A'}(y) &=
\frac{\IP_y[\tau_1(\partial A')<\tau_1(A)]}
{\Es_A^{A'} }.
\label{df_hm_set_RW}
\end{align}
Informally, it is the ``harmonic measure with respect to~$A'$'';
the ``real'' harmonic measure on~$A$
would be recovered as~$A'$
expands towards the whole space~$\Z^2$.
Similarly to Lemma~6.1 of~\cite{CT14}
 one can obtain the following important fact:
the measure
\begin{equation}
\label{df_psi}
\hat\pi(y,z) = \hm_A^{A'}(y)\IP_{y}[S_{\tau_1(\partial A')}=z]
\end{equation}
is invariant for the Markov chain
$(Z^{(j)}_{\text{st}}, Z^{(j)}_{\text{end}})$.

%

It is also important to have an expression for the average
interexcursion time (i.e., the time between the beginning
of two consecutive excursions, in the stationary regime).
\begin{lem}
\label{l_interexc_time}
 For the SRW on~$\Z^2_n$,
it holds that
\begin{equation}
\label{eq_interexc_time}
\IE (J_1-J_0\mid \Gamma_0)
= \IE (D_1-D_0\mid \Gamma_0)
= \frac{n^2}{\Es_A^{A'}}.
\end{equation}
\end{lem}

\begin{proof}
By the time-reversal invariance of the stationary walk,
for $x\in\partial A$ we have
$\IP[\Gamma_0,S_0=x]=n^{-2}\IP_x[\tau_1(\partial A')<\tau_1(A)]$, so
\begin{equation}
\label{prob_E0}
\IP[\Gamma_0]=\frac{\Es_A^{A'}}{n^2}
\qquad\text{and}\qquad
\IP[S_0=x\mid \Gamma_0]=\hm_A^{A'}(x).
\end{equation}
The point process of the starting times of the $A\to\partial A'$ excursions
is stationary and has intensity~$\IP[\Gamma_0]$, so the mean time
between two consecutive points is~$1/\IP[\Gamma_0]$
(and the same holds for the ending times, by stationarity).
This implies~\eqref{eq_interexc_time}.
\end{proof}

%

Next, assume that $B\subset B'\subset A\subset A'\subset\Z^2$
are such that $A\cap \partial A' = \emptyset$
and $B\cap \partial B' = \emptyset$.
Now, we need a result for controlling the
expected number of visits to~$B$
and the expected number of $B\to \partial B'$
excursions generated by one $A\to \partial A'$
excursion with the initial point chosen
according to $\hm_A^{A'}$:
\begin{lem}
\label{l_expected_generated}
With one $A\to \partial A'$
excursion with the initial point chosen
according to $\hm_A^{A'}$,
 we have that the expected number of visits to~$B$
 is equal to $\frac{|B|}{\Es_A^{A'}}$,
and the expected number of $B\to \partial B'$
excursions is equal to $\frac{\Es_B^{B'}}{\Es_A^{A'}}$.
\end{lem}

\begin{proof}
Since the quantities in the statement only depend on
the simple random walk before it hits~$\partial A'$,
we may place the (finite) set~$A'$ onto a torus~$\Z^2_n$ with
large enough~$n$, and work with the stationary simple random
walk on that torus; recall from~\eqref{prob_E0} that, in the stationary
regime, the starting point of an $A\to\partial A'$ excursion
has the law~$\hm_A^{A'}$.
Then, the walk is in~$B$ with frequency~$n^{-2}|B|$,
and the $A\to\partial A'$ (respectively, $B\to\partial B'$) excursions
start with frequency $n^{-2}\Es_A^{A'}$
(respectively, $n^{-2}\Es_B^{B'}$), by~\eqref{prob_E0}
applied to each of the two pairs.
Since all the visits to~$B$ and all the $B\to\partial B'$ excursions
occur during the $A\to\partial A'$ excursions,
both claims follow from the ergodic theorem.
\end{proof}


Next, it is important to understand the particular case
of excursions between (the boundaries of) two
concentric disks.
Lemma~2.5 of~\cite{CP17} implies that for
any~$x\in \partial \B(r)$
\[
 \IP_x[\tau_1(\partial \B(R))<\tau_1(\B(r))]
= \frac{\hm_{\B(r)}(x)}{\frac{2}{\pi}\ln\frac{R}{r}+O(r^{-1})}.
\]
Summing the above in~$x\in \partial \B(r)$,
we obtain that
\begin{equation}
\label{expr_esc_disks}
 \Es_{\B(r)}^{\B(R)}
 = \Big(\frac{2}{\pi}\ln\frac{R}{r}+O(r^{-1})\Big)^{-1}.
\end{equation}

Let us note that the above (together with
Lemma~\ref{l_expected_generated}) implies the
following fact, that will be repeatedly used in
the sequel: on average, one $n\to en$ excursion
generates approximately one $\B(x,r)\to\partial\B(x,er)$ excursion
provided that $\B(x,er)\subset\B(n)$ (of course,
we also assume that~$r$ is sufficiently large).
Moreover, the number
of $\B(x,r)\to\partial\B(x,er)$ excursions
generated by~$m$ independent $n\to en$ excursions
is (relatively) concentrated around~$m$ when~$m$ is much larger
than $\ln\frac{en}{r}$, see Lemma~\ref{l_exc_count_one_pair} below.

We also need the following fact about the excursion process generated by a sequence
of independent larger excursions.
\begin{lem}
\label{l_restriction_hm_indep}
Let $B\subset B'\subset A\subset A'\subset\Z^2$ with
$A\cap\partial A'=\emptyset$ and $B\cap\partial B'=\emptyset$.
Let $(\tilde Z^{(j)})_{j\ge1}$ be i.i.d.\ $A\to\partial A'$ excursions,
each started from $\hm_A^{A'}$ on $\partial A$.
From each $\tilde Z^{(j)}$, extract all $B\to\partial
B'$ excursions contained there.
Then the resulting infinite sequence of $B\to\partial B'$
excursions is a Markov chain whose stationary entrance law on
$\partial B$ is $\hm_B^{B'}$.
\end{lem}

\begin{proof}
The Markov property is a consequence of the strong Markov property
of the SRW and of the independence of the~$\tilde Z^{(j)}$'s.
Since the extracted sequence is a concatenation of i.i.d.\ blocks
(those generated by~$\tilde Z^{(1)},\tilde Z^{(2)},\ldots$),
its stationary entrance law gives to $y\in\partial B$ a weight
proportional to the expected number of $B\to\partial B'$ excursions
started at~$y$ generated by one $A\to\partial A'$ excursion;
analogously to the proof of Lemma~\ref{l_expected_generated},
the latter is equal to
$\hm_{B}^{B'}(y)\,\Es_{B}^{B'}/\Es_{A}^{A'}$.
%
%
%
%
%
\end{proof}

\subsection{Sharp control of excursion counts}
\label{s_exc_control}
Here, we formulate a couple of results that permit
us to control various excursion counts.
In the following, we refer to any pair of
concentric disks~$\B(x,r)$ and~$\B(x,R)$
(where $r+1<R$) as a pair of nested disks.

First, we need a result for the number of ``smaller''
excursions between certain sets generated by independent
excursions between $\B(n)$ and $\partial\B(en)$
(started at the relative harmonic measure~$\hm_{\B(n)}^{\B(en)}$).
Fix~$n$; for $m<n$,
\begin{equation}
\label{admissible_collection}
  B_i=\B\big(x_i,{\tilde\alpha} \tfrac{n}{m}\big),
 \quad
 B'_i=\B\big(x_i,{\tilde\alpha} e\tfrac{n}{m}\big),
 \qquad i=1,\ldots,m^2
\end{equation}
is said to be an \emph{admissible collection}
of nested disks if ${\tilde\alpha}\geq \alpha=0.1$,
$B'_i\cap B'_j=\emptyset$ for $i\neq j$,
and $B'_i\subset \B(n)\setminus\partial\B(n)$ for $i=1,\ldots,m^2$
(i.e., these pairs of nested disks do not intersect
and occupy
a positive proportion of the area of the bigger disk).
We have
\begin{lem}
\label{l_excurs_control_excursions}
Assume that we have an admissible collection
of disks as above.
 Consider a sequence of independent
excursions between $\B(n)$ and $\partial\B(en)$ started at the relative harmonic measure,
and let~$\eta_j$ be the total number of $B_i\to\partial B'_i$
excursions, $i=1,\ldots,m^2$, generated by the $j$th
excursion of that sequence. Then
we have
\begin{equation}
\label{eq_excurs_control_Eeta}
\IE\eta_1 = m^2 \big(1+O\big(\tfrac{m}{n}\big)\big)
\end{equation}
and, for
$\delta\in(0,1)$,
\begin{equation}
\label{eq_excurs_control_excursions}
\IP\big[(1-\delta)k\IE\eta_1 \leq \eta_1+\cdots+\eta_{k}
  \leq (1+\delta)k\IE\eta_1 \big]
   \geq 1 - 2e^{-c\delta^2 k}.
\end{equation}
\end{lem}

\begin{proof}
First, \eqref{eq_excurs_control_Eeta} follows from
Lemma~\ref{l_expected_generated} (applied to each of the pairs
$(B_i,B'_i)$) and~\eqref{expr_esc_disks}.
Next, for $z\in\B(en)$, let~$N$ be the number of $B_i\to\partial B'_i$
excursions, $i=1,\ldots,m^2$, made by the SRW started at~$z$ before it
hits~$\partial\B(en)$. Then $\sup_{z\in\B(en)}\IE_zN\leq Cm^2$:
indeed, each such excursion takes time at least $c(n/m)^2$ on average,
while the expected hitting time of~$\partial\B(en)$ is at most~$Cn^2$,
uniformly in the starting points. By the Markov inequality and the strong
Markov property (applied at the ends of these excursions), this implies
that $\IP[\eta_1\geq sm^2]\leq 2e^{-c's}$ for all $s\geq 0$.
Then, \eqref{eq_excurs_control_excursions} follows from the Bernstein inequality
(basically, $m^{-2}\eta_1$ is a mean-$1$ random variable with
exponential tail).
\end{proof}

We also need to control the number of excursions
between just one pair of nested disks, which can be placed
anywhere in~$\B(n)$:
\begin{lem}
\label{l_exc_count_one_pair}
Let $B=\B(x,r)$ and $B'=\B(x,er)$ be such that
$B'\subset\B(n)\setminus\partial\B(n)$ (and~$r$ is large enough).
Consider a sequence of independent
excursions between $\B(n)$ and $\partial\B(en)$
started at the relative harmonic measure,
and let~$\zeta_j$ be the number of $B\to\partial B'$ excursions
generated by the $j$th excursion of that sequence. Then
$\IE\zeta_1=1+O(r^{-1})$, and, for all $k\geq 1$ and $\delta>0$,
\begin{equation}
\label{eq_exc_count_one_pair}
\IP\big[|\zeta_1+\cdots+\zeta_k-k\IE\zeta_1|>\delta k\big]
 \leq 2\exp\Big(-\frac{c(\delta\wedge\delta^2)k}{\ln\frac{en}{r}}\Big).
\end{equation}
\end{lem}

\begin{proof}
The first claim follows from Lemma~\ref{l_expected_generated}
and~\eqref{expr_esc_disks}.
Since
$\B(x,(e-1)n)\subset\B(en)\subset\B(x,4n)$, Lemma~3.11
of~\cite{P21} implies that
$c\big(\ln\frac{en}{r}\big)^{-1}\leq
\IP_y\big[\tau_{\partial\B(en)}<\tau_{B}\big]
\leq C\big(\ln\frac{en}{r}\big)^{-1}$ for all $y\in\partial B'$.
Therefore, given that $\zeta_1\geq 1$, the random variable~$\zeta_1$
is dominated by a Geometric($c/\ln\frac{en}{r}$) random variable
and dominates a Geometric($C/\ln\frac{en}{r}$) one;
since $\IE\zeta_1\leq 2$, the latter fact implies that
$\IP[\zeta_1\geq 1]\leq 2C/\ln\frac{en}{r}$.
It is then elementary to obtain that
$\ln\IE e^{\pm\lambda(\zeta_1-\IE\zeta_1)}\leq C'\lambda^2\ln\frac{en}{r}$
for all $\lambda\in\big[0,c'/\ln\frac{en}{r}\big]$,
and (since the $\zeta_j$'s are i.i.d.)
we conclude with the exponential Chebyshev inequality,
used with $\lambda=c''(\delta\wedge 1)/\ln\frac{en}{r}$.
\end{proof}

\subsection{Excursions and soft local times}
\label{s_exc_SLT}
Here we recall the construction of excursions with the method of
soft local times~\cite{SLT} (see also Section~5.1 of~\cite{P21}),
and obtain large deviation estimates for the soft local times, which
permit us to compare ``real'' excursions with independent ones.
The setup is general enough to cover also the situations in the
companion paper; in particular, the excursions between the boundaries
of a pair of nested disks may be constructed
simultaneously with excursions elsewhere.

\medskip
\noindent
\textbf{Setup.} Let~$X$ be one of the following processes:
(a)~SRW on~$\Z^2$; (b)~SRW on the torus~$\Z^2_n$;
(c)~the concatenation of a sequence of independent
$N\to eN$ excursions (``big excursions''), each of them started
at~$\hm_{\B(N)}^{\B(eN)}$ (we write~$N$ because, in Section~\ref{s_proofs_main},
this case is used with several different radii of the big excursions).
The pair of interest is
\[
 A=\B(y_0,r), \qquad A'=\B(y_0,R),
\]
where $R\geq c_*r$ for a fixed $c_*>1$ (and $R<n/3$ in the
case~(b)).
In addition, simultaneously with the $A\to\partial A'$ excursions,
we may construct the excursions of~$X$ between~$\partial\widetilde A$
and~$\partial\widetilde A'$, where $\widetilde A\subset\widetilde A'$ are finite (possibly
empty) sets such that $\widetilde A\cap\partial\widetilde A'=\emptyset$
and $A'\cap\widetilde A'=\emptyset$ (this is needed to obtain independent
excursions in disjoint pairs of nested disks, see the last part
of Corollary~\ref{c_SLT_domination}).
Abbreviate $\A=A\cup\widetilde A$, $\A'=A'\cup\widetilde A'$, and\footnote{This is
an abuse of notation, since
the boundary of the set $A'\cup\widetilde A'$ can be strictly smaller
when~$A'$ and~$\widetilde A'$ are adjacent; note, however, that
$\partial\A=\partial A\cup\partial\widetilde A$.}
$\partial\A':=\partial A'\cup\partial\widetilde A'$.
In the case~(c) we assume that $\A'\subset\B(N)$ and
$\A'\cap\partial\B(N)=\emptyset$.
The starting point~$x_0$ of~$X$ is such that
$x_0\notin\A\cup(A'\setminus\partial A')$ (in the case~(c), $x_0$ is
the starting point of the first big excursion).

Let $Z^{(1)},Z^{(2)},\ldots$ be the successive excursions of~$X$
between~$\partial\A$ and~$\partial\A'$ (each of them is either
an $A\to\partial A'$ excursion or
a $\widetilde A\to\partial\widetilde A'$ excursion), and
set $Z^{(0)}_{\text{end}}:=x_0$. For $x\in\{x_0\}\cup\partial\A'$
and $y\in\partial\A$ let $\phi(x,y)=\IP_x[X_{\tau_1(\A)}=y]$
(in the case~(c), the current big excursion is continued from~$x$
and then followed by independent big excursions), and let
$\psi(x)=\sum_{y\in\partial A}\phi(x,y)$.
Given $Z^{(i-1)}_{\text{end}}=x$, the starting point of~$Z^{(i)}$
has the law~$\phi(x,\cdot)$, and, given also that it equals
$y\in\partial A$ (respectively, $y\in\partial\widetilde A$),
$Z^{(i)}$ is a SRW path started at~$y$ and
stopped upon hitting~$\partial A'$
(respectively, $\partial\widetilde A'$).

We construct these excursions with soft local times as in
Section~5.1 of~\cite{P21} (see also Figure~2 of~\cite{CP17}),
using a marked Poisson point process on~$\partial\A\times\R_+$ with
intensity (counting measure)$\,\otimes\,du$, where the independent
mark of $(z,u)$ is a SRW path started at~$z$
and run until hitting~$\partial A'$ (if $z\in\partial A$)
or~$\partial\widetilde A'$ (if $z\in\partial\widetilde A$). This yields the excursions
$Z^{(1)},Z^{(2)},\ldots$ together with random variables
$\xi_1,\xi_2,\ldots$, and the soft local time at the $t$th step is
\begin{equation}
\label{df_SLT_torus}
 G_t(y) = \sum_{i=1}^t \xi_i\,\phi\big(Z_{\text{end}}^{(i-1)},y\big),
 \qquad y\in\partial\A.
\end{equation}
Let $\F_i=\sigma(\xi_j,Z^{(j)},j\leq i)$.
We need the following properties of this construction:
\begin{prop}
\label{p_SLT_facts}
\begin{itemize}
 \item[(i)] The sequence $(Z^{(i)},i\geq 1)$ has the law of
 the successive excursions of~$X$
 between~$\partial\A$ and~$\partial\A'$.
 \item[(ii)] The random variables $\xi_1,\xi_2,\ldots$ are i.i.d.\
 Exponential($1$), and, for all $i\geq 1$, $\xi_i$ is independent of
 $\F_{i-1}\vee\sigma(Z^{(i)})$.
 \item[(iii)] For all $t\geq 1$, the points of the Poisson process
 lying below (or on) the graph of~$G_t$ are exactly those chosen
 in the first~$t$ steps.
\end{itemize}
\end{prop}

\begin{proof}
Items (i)--(ii) follow from Lemma~5.1 and
Proposition~5.3 of~\cite{P21}, and~(iii) holds by construction.
\end{proof}

When $\widetilde A=\widetilde A'=\emptyset$
we speak of the \emph{direct} construction of the
$A\to\partial A'$ excursions, and otherwise of the \emph{joint} one.
For $k\geq 1$ let~$T_k$ be the index of the $k$th excursion
between~$\partial A$ and~$\partial A'$ in~$(Z^{(i)})$, and let
\[
 L_k(y)=G_{T_k}(y), \quad y\in\partial A, \qquad
 M_k=\sum_{y\in\partial A}L_k(y)
 = \sum_{i=1}^{T_k}\xi_i\,\psi\big(Z_{\text{end}}^{(i-1)}\big)
\]
be the soft local time on~$\partial A$ at the moment when the
$k$th $A\to\partial A'$ excursion is chosen, and its total mass.
In the direct construction, $T_k=k$ and $M_k=\xi_1+\cdots+\xi_k$
(the ``intrinsic clock'' of the soft local time); in the joint one,
$M_k$ still has the same law (see Lemma~\ref{l_LD_SLT}~(i)).

\begin{rem}
\label{r_consistency}
Lemma~2.10 of~\cite{CP17} states that the law of the soft local
time~$L_k(\cdot)$ on~$\partial A$ is the same in the direct and in
the joint constructions. This is not correct in general
(only the laws of the excursions and of the total mass~$M_k$,
and the expectations $\IE L_k(y)$, are the same).
Indeed, in the direct construction,
$L_1(\cdot)=\xi_1\IP_{x_0}[X_{\tau_1(A)}=\cdot\,]$
is a.s.\ proportional to a deterministic measure on~$\partial A$,
while in the joint one $L_1(\cdot)=\xi_1\phi(x_0,\cdot)$ on $\{T_1=1\}$,
and $L_1(\cdot)=\xi_1\phi(x_0,\cdot)+\xi_2\phi(x,\cdot)$
on $\{T_1=2,Z^{(1)}_{\text{end}}=x\}$.
So, the two laws can only coincide if the restrictions
to~$\partial A$ of the measures~$\phi(x,\cdot)$,
for $x=x_0$ and for the possible endpoints $x\in\partial\widetilde A'$
of the first excursion, are all proportional to each other, i.e., if
the (normalized) entrance measure to~$A$ does not depend on the
starting point. This is not the case in general,
while~\eqref{kernel_flat} gives approximate independence of the
starting point when~$R/r$ is large.
In any case, we do not need this property, since
Lemma~\ref{l_LD_SLT} and
Corollary~\ref{c_SLT_domination} below
apply to the joint construction directly.
\end{rem}
In the next lemma, we collect some technical facts
about hitting probabilities.
\begin{lem}
\label{l_kernel_SLT}
There exist $r_0\geq 1$, $\beta\in(0,1]$, and positive constants
$c_5,C_5$ (depending only on~$c_*$) such that, for all $r\geq r_0$,
$x\in\{x_0\}\cup\partial\A'$, and $y,z\in\partial A$,
\begin{align}
 c_5\,\psi(x)\hm_A(y)\leq\phi(x,y)
 &\leq C_5\,\psi(x)\hm_A(y),
\label{kernel_Harnack}\\
 \Big|\frac{\phi(x,y)}{\psi(x)\hm_A(y)}-1\Big|
 &\leq C_5\,\frac{r}{R}\ln\frac{R}{r}\qquad\text{(if $\psi(x)>0$)},
\label{kernel_flat}\\
 \Big|\frac{\phi(x,y)}{\hm_A(y)}-\frac{\phi(x,z)}{\hm_A(z)}\Big|
 &\leq C_5\,\psi(x)\Big(\frac{\|y-z\|}{r}\Big)^\beta,
\label{kernel_Holder}\\
 \Big|\frac{\hm_A^{A'}(y)}{\hm_A(y)}-1\Big| &\leq\frac{C_5}{r}.
\label{hm_vs_hmrel}
\end{align}
Moreover, \eqref{kernel_Holder} remains valid (with a larger
constant~$C_5$) with~$\hm_A^{A'}$ in place of~$\hm_A$.
\end{lem}

\begin{proof}
Let $c_0=\min\{c_*,2\}$, $\rho_0=\frac{1+c_0}{2}$,
$\Lambda=\big(\B(y_0,c_0r)\setminus\B(y_0,r)\big)\cup\partial\B(y_0,r)$,
and let $H_\Lambda(v,\cdot)$ be the exit measure from~$\Lambda$ of
the SRW started at~$v$ (as in~(13) of~\cite{CP17}). Decomposing the
trajectory of~$X$ (started at~$x$) at its successive visits
to~$\partial\B(y_0,\rho_0r)$ that are followed by exits from~$\Lambda$
(note that $\|x-y_0\|>\rho_0r$, and that~$X$ is a SRW while inside
$\Lambda\subset A'$), we obtain that
$\phi(x,\cdot)|_{\partial A}=\sum_v\nu_x(v)H_\Lambda(v,\cdot)|_{\partial A}$
for a measure~$\nu_x$ on~$\partial\B(y_0,\rho_0r)$, whose total
mass is at most~$C\psi(x)$ since $H_\Lambda(v,\partial A)$ is bounded
away from~$0$ (by Lemma~3.11 of~\cite{P21}). Then~\eqref{kernel_Harnack}
and~\eqref{kernel_Holder} follow from Lemma~2.4 (together with the
observation made just after it) and Proposition~1.3 of~\cite{CP17},
used with~$r$ in place of~$n$ and $c=c_0$.

Let~$\gamma_0>2$ be a large enough constant.
If $R\leq\gamma_0r$, \eqref{kernel_flat} follows from~\eqref{kernel_Harnack}
(since then $\frac{r}{R}\ln\frac{R}{r}$ is bounded away from~$0$).
If $R>\gamma_0r$, let $D=\partial\B(y_0,R/2)$. After its last visit
to~$\partial A'$ before~$\tau_1(\A)$, the process~$X$ is a SRW
path in~$A'$ that has to cross~$D$; so the last-exit decomposition
shows that $\phi(x,\cdot)/\psi(x)$ is a convex combination of the
measures $\IP_z\big[S_{\tau_1(A)}\in\cdot\mid\tau_1(A)<\tau_1(\partial A')\big]$,
$z\in D$. For large enough~$\gamma_0$ and~$r$, Theorem~3.20 of~\cite{P21} applies to these
measures, with
$\Psi=\big(a(\dist(z,A))-\capa(A)\big)\vee1=\frac{2}{\pi}\ln\frac{R}{r}+O(1)$
by~(3.36) and~(3.66) of~\cite{P21}; this gives~\eqref{kernel_flat}.
Next, \eqref{hm_vs_hmrel} follows from Lemma~2.5 of~\cite{CP17}
and the definition~\eqref{df_hm_set_RW} of~$\hm_A^{A'}$.
Finally, by~\eqref{kernel_Harnack} and~\eqref{hm_vs_hmrel}, replacing~$\hm_A$
by~$\hm_A^{A'}$ in~\eqref{kernel_Holder} only produces an additional
term of order $\psi(x)/r$ in its right-hand side; for $y\neq z$, this term
is at most $C\psi(x)\big(\frac{\|y-z\|}{r}\big)^\beta$,
since $\|y-z\|\geq 1$ and $\beta\leq 1$.
\end{proof}

Now, let
\[
 \eps_{r,R}=\frac{r}{R}\ln\frac{R}{r}+\frac{1}{r}.
\]
The main result of this section is the following
\begin{lem}
\label{l_LD_SLT}
There exist positive constants $C_6,C_7,c_7$ (depending only
on~$c_*$) such that, for all $r\geq r_0$ and $k\geq 1$,
\begin{itemize}
 \item[(i)] the random variable~$M_k$ has the Gamma($k,1$) distribution;
 in particular, for all $t\in(0,k]$
 \begin{equation}
 \label{eq_LD_mass}
  \IP\big[|M_k-k|\geq t\big]\leq 2e^{-t^2/(4k)};
 \end{equation}
 \item[(ii)] for all $\theta\in(0,1]$
 \begin{equation}
 \label{eq_LD_SLT_rw}
 \IP\Big[\sup_{y\in\partial A}
 \Big|\frac{L_k(y)}{\hm_A^{A'}(y)}-M_k\Big|
 \geq \eps_{r,R}\Big(C_6\sqrt{k\big(1+\ln\tfrac{R}{r}\big)}+\theta k\Big)\Big]
 \leq C_7e^{-c_7\theta^2k}.
 \end{equation}
\end{itemize}
\end{lem}

\begin{proof}
Let $\psi_i=\psi\big(Z^{(i-1)}_{\text{end}}\big)$
and $\eta_i=\1{Z^{(i)}_{\text{st}}\in\partial A}$;
given~$\F_{i-1}$, $\eta_i$ is
Bernoulli($\psi_i$) and, by Proposition~\ref{p_SLT_facts}~(ii), independent of~$\xi_i$.

We first make the following observation. Let $(\zeta_i,i\geq 1)$ be
$\{0,1\}$-valued random variables such that~$\zeta_i$ is $\F_i$-measurable
and, given~$\F_{i-1}$, $\zeta_i$ is Bernoulli($c\psi_i$) (for a fixed
$c\in(0,1]$) and independent of~$\xi_i$. Let~$\tau$ be an a.s.\ finite
stopping time (with respect to~$(\F_i)$), and, for $j\geq 1$, let~$\tau^{(j)}$
be the $j$th index $i>\tau$ such that $\zeta_i=1$; assume that
$\tau^{(j)}<\infty$ a.s. Then
\begin{equation}
\label{eq_Gamma_mass}
 \text{conditionally on~$\F_\tau$, }\sum_{i=\tau+1}^{\tau^{(j)}}\xi_i\psi_i
 \text{ has the Gamma($j,c$) distribution}
\end{equation}
(i.e., that of the sum of~$j$ independent Exponential($c$) random variables).
Indeed, fix $s\geq 0$. Since
$\IE\big(e^{-s\xi_i\psi_i}\mid\F_{i-1}\big)=(1+s\psi_i)^{-1}$ and
$\IE\big((1+\frac{s}{c})^{\zeta_i}\mid\F_{i-1}\big)=1+s\psi_i$, the process
$\exp\big(-s\sum_{i=\tau+1}^t\xi_i\psi_i\big)
(1+\frac{s}{c})^{\sum_{i=\tau+1}^t\zeta_i}$, $t\geq\tau$,
is a martingale, which is bounded by~$(1+\frac{s}{c})^j$ up to
time~$\tau^{(j)}$. Since $\sum_{i=\tau+1}^{\tau^{(j)}}\zeta_i=j$, the optional
stopping theorem gives
$\IE\big(\exp\big(-s\sum_{i=\tau+1}^{\tau^{(j)}}\xi_i\psi_i\big)\mid\F_\tau\big)
=(1+\frac{s}{c})^{-j}$ for all $s\geq 0$, and~\eqref{eq_Gamma_mass} follows.

\smallskip
\noindent
\emph{Proof of~(i).} Since $T_k<\infty$ a.s.\ is the $k$th index~$i$
with $\eta_i=1$, \eqref{eq_Gamma_mass} (used with $\zeta_i=\eta_i$, $c=1$,
$\tau=0$, and $j=k$) shows that~$M_k$ has the Gamma($k,1$) distribution;
then, \eqref{eq_LD_mass} is the usual large deviation bound
for this distribution.

\smallskip
\noindent
\emph{Proof of~(ii).} We follow the strategy of the proof of
Lemma~2.9 of~\cite{CP17} (split the sequence of excursions into
independent blocks, and apply empirical processes' tools to the
soft local times of these blocks), so we mainly indicate
what has to be modified.

By~\eqref{kernel_Harnack} and~\eqref{hm_vs_hmrel}, there is
$\lambda>0$ such that $\phi(x,y)\geq\lambda \psi(x)\hm_A^{A'}(y)$
for all~$x$ and $y\in\partial A$. When the $i$th excursion starts
at $y\in\partial A$, we toss an independent coin~$\chi_i$ with
success probability
$\lambda \psi_i\hm_A^{A'}(y)/\phi(Z^{(i-1)}_{\text{end}},y)$
(and set $\chi_i=0$ when $\eta_i=0$); we include $\chi_1,\ldots,\chi_i$ in~$\F_i$ (clearly,
Proposition~\ref{p_SLT_facts}~(ii) remains valid). Then, given~$\F_{i-1}$ and
$\{\eta_i=1\}$, $\chi_i$ is Bernoulli($\lambda$) and, on
$\{\chi_i=1\}$, $Z^{(i)}_{\text{st}}$ has the law~$\hm_A^{A'}$;
in particular, given~$\F_{i-1}$, $\chi_i$ is Bernoulli($\lambda\psi_i$)
and independent of~$\xi_i$.
Let $t_1<t_2<\ldots$ be the indices with $\chi_i=1$, $t_0=0$,
and, for $b\geq 0$,
\[
 W_b(y)=\sum_{i=t_b+1}^{t_{b+1}}\xi_i\,
    \phi\big(Z^{(i-1)}_{\text{end}},y\big),\ y\in\partial A,
 \qquad
 \mu_b=\sum_{y\in\partial A}W_b(y)=\sum_{i=t_b+1}^{t_{b+1}}\xi_i\psi_i.
\]
Then $W_1,W_2,\ldots$ are i.i.d.\ (and independent of~$W_0$).
Also, since each $A\to\partial A'$ excursion is a renewal with
probability~$\lambda$, independently of the past, the renewals a.s.\
occur infinitely often; so, \eqref{eq_Gamma_mass} (used with
$\zeta_i=\chi_i$, $c=\lambda$, $j=1$, and $\tau=t_b$ or $\tau=T_k$)
implies that~$\mu_b$ has the Exponential($\lambda$) distribution
for all $b\geq 0$, and that, conditionally on~$\F_{T_k}$, so does
the total mass~$\mu'_k$ of the soft local time accumulated between~$T_k$
and the next renewal.

For $y\in\partial A$ and $w\in\R_+^{\partial A}$ let
$f_y(w)=\frac{w(y)}{\hm_A^{A'}(y)}-\sum_{z}w(z)$.
We have $\IE f_y(W_1)=0$: by Proposition~\ref{p_SLT_facts}~(ii),
$\IE W_1(y)$ and $\IE\mu_1$ are the expected numbers of
$A\to\partial A'$ excursions in a block that start at~$y$ and
anywhere, respectively, and, by the renewal-reward theorem, their
ratio is the stationary entrance law of these excursions, which
is~$\hm_A^{A'}$ by~\eqref{df_psi} or Lemma~\ref{l_restriction_hm_indep}.
Also, by~\eqref{kernel_flat} and~\eqref{hm_vs_hmrel},
$|f_y(w)|\leq C\eps_{r,R}\sum_zw(z)$ for all~$w$ of the form
$\sum_js_j\phi(x_j,\cdot)$, $s_j\geq0$. Then, denoting
\[
 S_m(y)=\sum_{b=1}^m f_y(W_b),
\]
we obtain (note that at most~$k$ renewals occur before~$T_k$)
\begin{equation}
\label{SLT_decomp}
 \sup_{y\in\partial A}\Big|\frac{L_k(y)}{\hm_A^{A'}(y)}-M_k\Big|
 \leq C\eps_{r,R}(\mu_0+\mu'_k)
  + \max_{m\leq k}\sup_{y}|S_m(y)|.
\end{equation}
Since~$\mu_0$ and~$\mu'_k$ are Exponential($\lambda$) (the latter
conditionally on~$\F_{T_k}$), the first term in the right-hand side
exceeds $\frac{1}{2}\eps_{r,R}\theta k$ with probability at
most $2e^{-c\theta k}$.

It remains to deal with the second term. Abbreviate
$\Delta=1+\ln\frac{R}{r}$, and let $h_y=\eps_{r,R}^{-1}f_y$
(recall that $\IE f_y(W_1)=0$). By the above,
$F(w)=C\sum_zw(z)$ is an envelope for
$\{h_y,y\in\partial A\}$, and, by~\eqref{kernel_Holder}
with~$\hm_A^{A'}$ in place of~$\hm_A$ (see Lemma~\ref{l_kernel_SLT}),
$|h_z-h_{z'}|(W_1)\leq\frac{C}{\eps_{r,R}}\big(\frac{\rho}{r}\big)^\beta
\mu_1$ for $\|z-z'\|\leq\rho$. So, as in Step~3 of the
proof of Lemma~2.9 of~\cite{CP17}, the bracketing numbers satisfy
$N_{[\,]}(s\|F\|_2,\{h_y\},\|\cdot\|_2)\leq C(s\eps_{r,R})^{-1/\beta}$,
and the bracketing entropy integral is at most
$C(1+\ln(1/\eps_{r,R}))^{1/2}\leq C\Delta^{1/2}$. Then (1.2)
of~\cite{vdVW11} gives
\begin{equation}
\label{E_sup_S}
 \IE\sup_{y\in\partial A}\Big|\sum_{b=1}^mh_y(W_b)\Big|
 \leq C\sqrt{m\Delta}
 \quad\text{for all }m\geq 1.
\end{equation}
For the deviations, we use truncation together with Bousquet's
inequality~\cite{B02}; this permits us to avoid an additional term
of the form $e^{-c\theta k/\ln k}$ that comes from Adamczak's
inequality in~(51) of~\cite{CP17}. Let $u=C'\ln\frac{2}{\theta}$
with~$C'$ large enough so that
$\IE F(W_1)\1{F(W_1)>u}\leq\frac{\theta}{8}$. The contribution
of $h_y\1{F>u}$ is bounded by
$\sum_{b\leq m}\big(F(W_b)\1{F(W_b)>u}-\IE F(W_1)\1{F(W_1)>u}\big)
+\frac{\theta m}{4}$, and the Bernstein inequality shows that the
sum exceeds $\frac{\theta m}{8}$ with probability at most
$e^{-c\theta^2m}$. The centred functions
$\tilde h_y=h_y\1{F\leq u}-\IE h_y\1{F\leq u}(W_1)$ are bounded
by~$2u$, have variances at most~$C$, and satisfy the analogue
of~\eqref{E_sup_S} (uniformly in~$u$); so Bousquet's inequality
(applied to $\{\pm\tilde h_y\}$ with the deviation parameter
$v=c\theta^2m$) shows that, with probability at least $1-e^{-c\theta^2m}$,
$\sup_y|\sum_{b\leq m}\tilde h_y(W_b)|\leq 2C\sqrt{m\Delta}+\frac{\theta m}{8}$
(for small~$c$, since $\theta\ln\frac{2}{\theta}\leq 1$). Hence
\[
 \IP\Big[\sup_{y\in\partial A}\Big|\sum_{b=1}^mh_y(W_b)\Big|
 \geq C\sqrt{m\Delta}+\frac{\theta m}{2}\Big]\leq 2e^{-c\theta^2m},
\]
and the same bound (with modified constants) holds for the maximum
over $m\leq k$, by a standard stopping argument
(cf.\ the proof of~(40) of~\cite{CP17}): indeed, by~\eqref{E_sup_S}
and the Markov inequality,
$\sup_{y}\big|\sum_{b=m+1}^kh_y(W_b)\big|\leq 2C\sqrt{k\Delta}$
with probability at least~$\frac{1}{2}$, independently
of $W_1,\ldots,W_m$.
Since $S_m(y)=\eps_{r,R}\sum_{b\leq m}h_y(W_b)$, this concludes the
proof of Lemma~\ref{l_LD_SLT}.
\end{proof}

\begin{cor}
\label{c_SLT_domination}
Let $(\tilde\xi_j,j\geq 1)$ and $(\tilde Z^{(j)},j\geq 1)$ be
obtained from the same marked Poisson process
restricted to $\partial A\times\R_+$ by the soft local times
construction with the kernel~$\hm_A^{A'}(\cdot)$, which does not depend
on the endpoint of the previous excursion; we call it
the \emph{free} soft local times construction on~$\partial A$
(the $\tilde Z^{(j)}$'s are then i.i.d.\ $A\to\partial A'$ excursions
with initial law~$\hm_A^{A'}$). Also, let $Z^{A,(1)},Z^{A,(2)},\ldots$ be the
$A\to\partial A'$ excursions of the sequence~$(Z^{(i)})$. There exist
positive constants $C_8,C_9,c_9$ (depending only on~$c_*$) such that,
for all $r\geq r_0$, $k\geq 1$, and
$t\in\big[C_8\sqrt{k(1+\ln\frac{R}{r})},k\big]$,
\begin{align}
\lefteqn{
\IP\Big[\big\{\tilde Z^{(1)},\ldots,\tilde Z^{(\lfloor k-t\rfloor)}\big\}
 \subset\big\{Z^{A,(1)},\ldots,Z^{A,(k)}\big\}
 \subset\big\{\tilde Z^{(1)},\ldots,\tilde Z^{(\lceil k+t\rceil)}\big\}\Big]}
 \phantom{*********************************}
 \nonumber \\
 &\geq 1-C_9e^{-c_9t^2/k}.
\label{eq_SLT_domination}
\end{align}
Moreover, assume that $\widetilde A=\bigcup_{j=1}^JA_j$ and
$\widetilde A'=\bigcup_{j=1}^JA'_j$, where $(A_j,A'_j)$ are pairs
of nested disks satisfying the same assumptions as~$(A,A')$, with
pairwise disjoint~$A'_j$'s (and
with~$\partial\widetilde A'$ understood as $\bigcup_{j}\partial A'_j$).
Then~\eqref{eq_SLT_domination} also holds for each~$(A_j,A'_j)$
in place of~$(A,A')$, and the corresponding $J+1$ sequences
of independent excursions are independent.
\end{cor}

\begin{proof}
By Proposition~\ref{p_SLT_facts}~(iii), the excursions $Z^{A,(1)},\ldots,Z^{A,(k)}$ are the marks
of the Poisson points below the graph of~$L_k$, and
$\tilde Z^{(1)},\ldots,\tilde Z^{(j)}$ are the marks of those below
the graph of $(\tilde\xi_1+\cdots+\tilde\xi_j)\hm_A^{A'}(\cdot)$.
So (cf.\ (5.8)--(5.9) of~\cite{P21}) it is enough to note that, by
Lemma~\ref{l_LD_SLT} (with $\theta=\frac{t}{8k}$; recall that
$\eps_{r,R}\leq 1$) and the usual
large deviation bound for sums of i.i.d.\ Exponentials, with
probability at least $1-C_9e^{-c_9t^2/k}$ we have
$\tilde\xi_1+\cdots+\tilde\xi_{\lfloor k-t\rfloor}\leq k-\frac{t}{2}
\leq L_k(y)/\hm_A^{A'}(y)\leq k+\frac{t}{2}\leq
\tilde\xi_1+\cdots+\tilde\xi_{\lceil k+t\rceil}$ for all $y\in\partial A$.
The last statement holds because the above construction does not
depend on which of the pairs is the pair of interest, and the
independent excursions are
constructed from the restrictions of the Poisson process to the
disjoint sets $\partial A\times\R_+$ and $\partial A_j\times\R_+$,
$j=1,\ldots,J$.
\end{proof}

Compared to Lemma~2.9 of~\cite{CP17}, Lemma~\ref{l_LD_SLT} applies
to the joint construction, holds for all $\theta\in(0,1]$ without an
additional term of the form $e^{-c\theta k/\ln k}$, and takes into
account that the entrance measure to~$A$ is already ``almost flat''
when~$R/r$ is large (the logarithmic factor in~\eqref{kernel_flat} is
due to the conditioning, and cannot be removed in general, cf.\
Exercise~3.36 of~\cite{P21}). In this paper, we only need it with
$R=er$, in combination with the excursion count control of
Section~\ref{s_exc_control}: first, we show (e.g.\ with
Lemma~\ref{l_exc_count_one_pair}) that the number of smaller excursions
generated by the larger ones exceeds (or does not exceed) some~$k$,
and then use Corollary~\ref{c_SLT_domination} to compare these
smaller excursions with independent ones. For one pair of nested
disks, this combination can be stated in the following way.
\begin{cor}
\label{c_count_SLT}
Let $B=\B(x,r)$ and $B'=\B(x,er)$ be such that
$B'\subset\B(N)\setminus\partial\B(N)$. Consider~$v$ independent
$N\to eN$ excursions started at~$\hm_{\B(N)}^{\B(eN)}$, let~$\ZZ$
be the set of all $B\to\partial B'$ excursions generated by them,
and let~$\widetilde\ZZ(j)$ be the set of the first~$\lceil j\rceil$
independent excursions of the free soft local times construction
on~$\partial B$ (as in Corollary~\ref{c_SLT_domination}, with $A=B$
and $A'=B'$). There exist universal positive constants
$r_1,C_{10},C_{11},c_{11}$ such that, for all $r\geq r_1$ and all~$t$
such that $C_{10}\big(\sqrt{v}+\frac{v}{r}\big)\leq t\leq v$,
\begin{equation}
\label{eq_count_SLT}
 \IP\big[\widetilde\ZZ(v-t)\subset\ZZ\subset\widetilde\ZZ(v+t)\big]
 \geq 1-C_{11}\exp\Big(-c_{11}\frac{t^2}{v\ln\frac{eN}{r}}\Big).
\end{equation}
Moreover, if $(B_j,B'_j)$, $j=1,\ldots,J$, are pairs of nested disks
as above with pairwise disjoint~$B'_j$'s, then~\eqref{eq_count_SLT}
holds for each of them (with the joint construction of
Corollary~\ref{c_SLT_domination}), and the corresponding~$J$ sequences
of independent excursions are independent.
\end{cor}

\begin{proof}
Let~$\zeta$ be the cardinality of~$\ZZ$. By Lemma~\ref{l_exc_count_one_pair}
(with~$N$ in place of~$n$, and $\delta=\frac{t}{4v}$), with probability at least
$1-2\exp\big(-c\frac{t^2}{v\ln(eN/r)}\big)$ we have
$|\zeta-v\IE\zeta_1|\leq\frac{t}{4}$, and hence $|\zeta-v|\leq\frac{t}{2}$
(since $\IE\zeta_1=1+O(r^{-1})$ and $t\geq C_{10}v/r$, with
large enough~$C_{10}$). Since the elements of~$\ZZ$ are the
first~$\zeta$ excursions of the sequence $Z^{B,(1)},Z^{B,(2)},\ldots$
of Corollary~\ref{c_SLT_domination}, on that event we have
\[
 \big\{Z^{B,(1)},\ldots,Z^{B,(\lceil v-t/2\rceil)}\big\}
 \subset\ZZ\subset
 \big\{Z^{B,(1)},\ldots,Z^{B,(\lfloor v+t/2\rfloor)}\big\}.
\]
We then apply Corollary~\ref{c_SLT_domination} (with $R=er$) twice,
with $k=\lceil v-\frac{t}{2}\rceil$ and $k=\lfloor v+\frac{t}{2}\rfloor$,
and with~$\frac{t}{3}$ in place of~$t$ in both cases; note that
$\frac{t}{3}\in\big[C_8\sqrt{2k},k\big]$ if~$C_{10}$ is large enough, and that
$\lfloor k-\frac{t}{3}\rfloor\geq\lceil v-t\rceil$ (respectively,
$\lceil k+\frac{t}{3}\rceil\leq\lceil v+t\rceil$), since $t\geq C_{10}\geq 12$.
The last statement follows from the last part of
Corollary~\ref{c_SLT_domination}, since the excursion counts do not
depend on the construction.
\end{proof}

\section{Proofs of the main results}
\label{s_proofs_main}
In the proofs below, we repeatedly use the scheme described
at the end of Section~\ref{s_exc_SLT}: there are independent excursions
between the boundaries of a pair of nested disks
(usually, but not always, $\B(n)$ and $\B(en)$), which we call
``big excursions'', and a few smaller pairs of nested
disks~$(B_i,B'_i)$ are placed into the inner disk.
We adopt the following notational convention:
we denote by~$\ZZ_i$ the set of all $B_i\to\partial B'_i$
excursions generated by the big excursions (we call them
\emph{real} excursions, as opposed to the independent ones
introduced below), and by~$\ZZ_i(k)$
the set of the first~$\lfloor k\rfloor$ of them. Also,
$\widetilde\ZZ_i(k)$ stands for the set of the
first~$\lceil k\rceil$ i.i.d.\ $B_i\to\partial B'_i$
excursions with initial law $\hm_{B_i}^{B'_i}$, obtained by the
free soft local times construction on~$\partial B_i$ (from the same
Poisson process), as in
Corollary~\ref{c_SLT_domination}; we refer to these as
the independent excursions of the free soft local times
construction. Recall that, by Corollary~\ref{c_SLT_domination},
these independent excursions are independent for different~$i$.

\begin{proof}[Proof of Theorem~\ref{t_exc_noncovering}.]
In the following, we work with increasing events;
the case of decreasing events is analogous, and we indicate
the necessary modifications at the end of the proofs of the lower
and of the upper bounds.
To prove the lower bound~\eqref{eq_exc_noncovering_lower},
recall that $g_n\leq (1-\eps_1)k'_n$.
Each big (i.e., $n\to en$) excursion
generates a random number
of $n\to e^{1/2}n$ excursions (``small excursions''),
which has approximately Geometric($1/2$)
distribution, see Figure~\ref{f_generate_exc}.
\begin{figure}
\begin{center}
\includegraphics{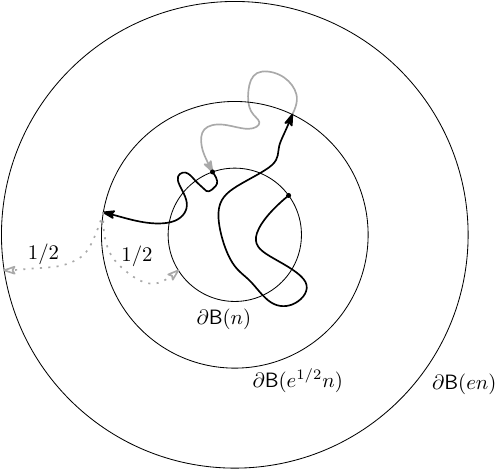}
\caption{Generating new $n\to e^{1/2}n$
excursions.}
\label{f_generate_exc}
\end{center}
\end{figure}
Since~$\B(n)$ can only be visited during small excursions,
the local times on~$\B(n)$ are the sums of the local times
generated by the small excursions. The idea is to split
the sequence of small excursions into i.i.d.\ blocks,
in such a way that the number of big excursions
affects the \emph{number} of blocks, but not their contents;
then, to make the (increasing) event~$E^{\B(n)}$ occur
with just $k'_n-g_n$ big excursions, it is enough to pay
for having an atypically large number of blocks.
This is reminiscent of the approach of~\cite{BGP18}
(see Section~3.1 there), where the regeneration times of
a Markov chain are determined by an independent sequence
of coin tosses, and the regeneration blocks of size one
are constructed after all the others.

Abbreviate $\mu_0:=\hm_{\B(n)}^{\B(en)}$, and, for
$y\in\partial\B(e^{1/2}n)$ and $z\in\partial\B(n)$, let
\[
 q(y)=\IP_y\big[\tau_{\partial\B(en)}<\tau_{\B(n)}\big],
 \qquad
 K(y,z)=\IP_y\big[S_{\tau_{\B(n)}}=z \mid
   \tau_{\B(n)}<\tau_{\partial\B(en)}\big].
\]
There exist $C_0>0$ and $p\in(0,1)$ (not depending on~$n$)
such that
\begin{equation}
\label{eq_q_half}
 \big|q(y)-\tfrac{1}{2}\big| \leq \frac{C_0}{n}
 \qquad \text{and} \qquad
 K(y,z)\geq p\,\mu_0(z)
\end{equation}
for all $y\in\partial\B(e^{1/2}n)$ and $z\in\partial\B(n)$.
Indeed, the first inequality follows from Lemma~3.11
of~\cite{P21}, and the second one from the Harnack inequality
(see e.g.\ Chapter~6 of~\cite{LL10}) and~\eqref{hm_vs_hmrel}.

Fix $\sigma\in\big(\frac{2+p}{4},\frac{1+p}{2}\big)$.
Since $q(y)+p(1-q(y))\geq\frac{1+p}{2}-\frac{C_0}{n}$
by~\eqref{eq_q_half}, for all large enough~$n$ we have
$q(y)<\sigma<q(y)+p(1-q(y))$ for all $y\in\partial\B(e^{1/2}n)$,
so (again by~\eqref{eq_q_half})
\[
 \nu_y(\cdot)=\frac{(1-q(y))K(y,\cdot)-(\sigma-q(y))\mu_0(\cdot)}{1-\sigma}
\]
is a probability measure on~$\partial\B(n)$.
Consider a sequence of independent big excursions, and let
$\gs_1,\gs_2,\gs_3,\ldots$ be the small excursions generated
by them, enumerated consecutively (first those generated
by the first big excursion, then those generated by the
second one, and so on); let $y_i\in\partial\B(e^{1/2}n)$
be the endpoint of~$\gs_i$.
By the strong Markov property, given $\gs_1,\ldots,\gs_i$,
the big excursion that generated~$\gs_i$ ends with
probability~$q(y_i)$ (and then~$\gs_{i+1}$, the first small
excursion of the next big excursion, starts from~$\mu_0$),
and otherwise~$\gs_{i+1}$ starts from~$K(y_i,\cdot)$;
given its starting point, $\gs_{i+1}$ is a simple random walk
stopped upon hitting~$\partial\B(e^{1/2}n)$.
We realize these transitions in the following way: after~$\gs_i$
is completed, we toss an independent coin which comes up heads
with probability~$\sigma$;
\begin{itemize}
 \item if it comes up tails, the current big excursion continues,
 and the starting point of~$\gs_{i+1}$ is chosen according
 to~$\nu_{y_i}$;
 \item if it comes up heads, the starting point of~$\gs_{i+1}$
 is chosen according to~$\mu_0$, and, independently, the current
 big excursion ends with probability~$q(y_i)/\sigma$
 (and continues otherwise).
\end{itemize}
This does not change the law of $(\gs_i)_{i\geq 1}$ together with
the information on which small excursions are the last ones of
their big excursions: indeed, the current big excursion ends with
probability $\sigma\cdot\frac{q(y_i)}{\sigma}=q(y_i)$, and the
probability that it continues and~$\gs_{i+1}$ starts at
$z\in\partial\B(n)$ is
$(1-\sigma)\nu_{y_i}(z)+(\sigma-q(y_i))\mu_0(z)=(1-q(y_i))K(y_i,z)$.

Let us call~$i$ a \emph{regeneration index} if $i=1$ or the coin
tossed after~$\gs_{i-1}$ came up heads, and let the $b$th
\emph{block} consist of the small excursions with indices
between the $b$th (included) and the $(b+1)$th (excluded)
regeneration indices. Let $\ell^{(b)}\in\Z_+^{\B(n)}$ be
the field of local times on~$\B(n)$ generated by the small
excursions of the $b$th block, and let~$Y_b$ be the endpoint
of the last of them. Each block starts afresh from~$\mu_0$,
whether or not the preceding block was the last one of a big
excursion; so, the blocks are i.i.d., and we can (and will)
make the decisions at the ends of the blocks using a sequence
$(U_b)_{b\geq 1}$ of i.i.d.\ Uniform$[0,1]$ random variables,
independent of the blocks: the $b$th block ends the current
big excursion if $U_b\leq q(Y_b)/\sigma$.
Denote $\Phi(b)=\ell^{(1)}+\cdots+\ell^{(b)}$ (note that $\Phi(b)$
is nondecreasing in~$b$), and let~$V_v$ be the index of the block
that ends the $v$th big excursion; then, the local times
on~$\B(n)$ generated by~$v$ big excursions are given by~$\Phi(V_v)$.

Next, let $q_\pm=\frac{1}{2}\pm\frac{C_0}{n}$ (note that $q_+<\sigma$
for large~$n$), and let~$V^\pm_v$ be the $v$th index~$b$ such that
$U_b\leq q_\pm/\sigma$. Since $q_-\leq q(Y_b)\leq q_+$
by~\eqref{eq_q_half}, we have
\[
 V^+_v\leq V_v\leq V^-_v \qquad\text{for all }v\geq 1.
\]
Note that, while~$V_v$ depends on the blocks (through the~$Y_b$'s),
$(V^\pm_v)_{v\geq 1}$ are independent of $(\ell^{(b)})_{b\geq 1}$,
and~$V^\pm_v$ has the negative binomial distribution
\begin{equation}
\label{eq_B_r_negbin}
 \IP\big[V^\pm_v=b\big]=\binom{b-1}{v-1}\Big(\frac{q_\pm}{\sigma}\Big)^v
 \Big(1-\frac{q_\pm}{\sigma}\Big)^{b-v},
 \qquad b\geq v;
\end{equation}
in particular, the expectation of~$V^\pm_v$ is
$2\sigma v\big(1+O(n^{-1})\big)$, and its variance is at most~$4v$.

Let us fix $b^*=\big\lceil 2\sigma k'_n+11\sqrt{k'_n}\,\big\rceil$.
Recall that Condition~L implies that $k'_n\to\infty$ and
$k'_n/n=o\big(\sqrt{k'_n}\big)$. Using the definition of~$k'_n$,
the inequality $V_{k'_n}\leq V^-_{k'_n}$, and the Chebyshev inequality
for~$V^-_{k'_n}$, we obtain that, for all large enough~$n$,
\begin{equation}
\label{eq_Phi_bstar}
 \IP\big[\Phi(b^*)\in E^{\B(n)}\big]
 \geq \IP\big[\Phi(V_{k'_n})\in E^{\B(n)}\big]
  -\IP\big[V^-_{k'_n}>b^*\big]
 \geq \frac{4}{5}
\end{equation}
(here, we also use that~$E^{\B(n)}$ is increasing and $\Phi(\cdot)$
is nondecreasing).

Now, let $v=k'_n-g_n$. On the event
$\{\Phi(b^*)\in E^{\B(n)}\}\cap\{V^+_v\geq b^*\}$ we have
$V_v\geq b^*$, so $\Phi(V_v)\geq \Phi(b^*)$, and therefore
$E^{\B(n)}$ occurs. Since these two events are independent,
\eqref{eq_Phi_bstar} implies that
\begin{equation}
\label{eq_lower_via_Br}
 \IP^{(k'_n-g_n)}\big[E^{\B(n)}\big]
 \geq \frac{4}{5}\,\IP\big[V^+_v\geq b^*\big].
\end{equation}

It remains to bound $\IP[V^+_v\geq b^*]$ from below;
note that we only pay for~$V^+_v$ exceeding its expectation
by approximately $2\sigma g_n$.
Let $\mathcal W=\big\{b^*,\ldots,b^*+\big\lfloor\sqrt{k'_n}\big\rfloor\big\}$.
For $b\in\mathcal W$ we have
$\big|(b-1)\frac{q_+}{\sigma}-(v-1)\big|\leq g_n+13\sqrt{k'_n}$
and $b\leq 3k'_n$
(recall also that $v\geq \eps_1k'_n$), so it is elementary to obtain
from~\eqref{eq_B_r_negbin} and Stirling's formula that
\[
 \IP\big[V^+_v=b\big] \geq \frac{c}{\sqrt{k'_n}}
 \exp\Big(-c_1\frac{g_n^2}{k'_n}\Big)
 \qquad\text{for all }b\in\mathcal W.
\]
Since $|\mathcal W|\geq\sqrt{k'_n}$, together
with~\eqref{eq_lower_via_Br} this implies
the lower bound~\eqref{eq_exc_noncovering_lower}.

In the case of decreasing events, one uses the same
construction with $v=\underline k_n+g_n$,
$b_*=\lfloor 2\sigma\underline k_n-11\sqrt{\underline k_n}\rfloor$
in place of~$b^*$, and
$\mathcal W=\{b_*-\lfloor\sqrt{\underline k_n}\rfloor,
\ldots,b_*\}$; now, \eqref{eq_Phi_bstar} is replaced by
$\IP[\Phi(b_*)\in E^{\B(n)}]\geq\IP[\Phi(V_{\underline k_n})\in E^{\B(n)}]
-\IP[V^+_{\underline k_n}<b_*]\geq\frac{4}{5}$,
and, since~$E^{\B(n)}$ is decreasing, it occurs
on $\{\Phi(b_*)\in E^{\B(n)}\}\cap\{V^-_v\leq b_*\}$.
Note, however, that $V^-_v\geq v$, so this argument requires
that $g_n\leq (2\sigma-1)\underline k_n-C\sqrt{\underline k_n}$;
since $2\sigma-1>\frac{p}{2}$, this holds for all large enough~$n$
when $g_n\leq\eps_3\underline k_n$ with $\eps_3:=\frac{p}{2}$.

\bigskip

Now, let us prove the upper bound~\eqref{eq_exc_noncovering_upper}.
Let~$K_0$ be a large constant (depending on~$\eps_1$ and~$\eps_2$),
to be chosen later. If $g_n<K_0\sqrt{k_n}$, there is nothing to prove:
since $\IP^{(j)}[E^{\B(n)}]$ is nondecreasing in~$j$, by~\eqref{def_kn}
we have
\[
 \IP^{(k_n-g_n)}[E^{\B(n)}] \leq \IP^{(k_n)}[E^{\B(n)}]
 \leq \frac{1}{7} \leq \exp\Big(-c_2 \frac{g_n^2}{k_n}\Big),
\]
provided that $c_2\leq K_0^{-2}\ln 7$ (note that~\eqref{eq_regularity}
is not needed here).
So, from now on we assume that $g_n\geq K_0\sqrt{k_n}$; in the
estimates below, we use (mostly without mentioning it explicitly)
that $g_n/\sqrt{k_n}$ is larger than a suitable constant, e.g., for
the quantities~$m$ and~$s$ defined below to be at least~$2$, and to absorb
the factors which are polynomial in $g_n/\sqrt{k_n}$ into the exponential terms.
Fix $\beta_1,\beta_2$ such that
$\frac{1}{2}<\beta_2<\beta_1<1$.
Define, with small enough~$c$
(recall~\eqref{eq_regularity})
\[
 m := c\Big(\frac{2\alpha g_n}{\sqrt{k_n}}\Big)^{\eps_2} .
\]
Note that, by~\eqref{eq_regularity} (used with $h=\frac{m}{\alpha}$), we have
$k_n-\beta_1 g_n \leq k_n-\frac{1}{2}g_n \leq k_{\alpha n/m}$.
Here and below, the numbers of excursions of the
form $k_n-\beta g_n$ need not be integers, and they are rounded
as explained in the beginning of this section; since the thresholds are
integers, the comparisons with the thresholds are not affected by that
(e.g., the above implies that
$\lceil k_n-\beta_1 g_n\rceil \leq k_{\alpha n/m}$).
Next, 
consider an
admissible collection of pairs $(B_i,B'_i)$,
$i=1,\ldots,m^2$, as in~\eqref{admissible_collection}
(with $\tilde\alpha=\alpha$).
We say that $(B_i,B'_i)$ is \emph{good},
if, in the free soft local times construction
on~$\partial B_i$,
$k_n-\beta_1 g_n$ independent excursions
(i.e., those of $\widetilde\ZZ_i(k_n-\beta_1g_n)$)
make the event $(E^{B_i})^\complement$ occur.
Let $G_i=\{(B_i,B'_i) \text{ is good}\}$, and note
that the events $G_1,\ldots,G_{m^2}$ are independent
(by the last part of Corollary~\ref{c_SLT_domination}).

Denote
\[
 \psi_{n,m} = 1-\IP\big[(B_i,B'_i)\text{ is good}\big]
 = \IP_{\alpha\frac{n}{m}}^{(\lceil k_n-\beta_1 g_n\rceil)}
 \big[E^{\B(\alpha\frac{n}{m})}\big].
\]
We now obtain an upper bound on~$\psi_{n,m}$.
The idea is to place~$s^2$ smaller nested disks
of (inner disks) radii~$\frac{\alpha^2 n}{ms}$
into~$B_1$, where
$s=\frac{2\alpha g_n}{\sqrt{k_n}}$
(see Figure~\ref{f_two_scales}) and
then use the soft local time construction
for generating excursions on these smaller
nested disks (that is, now $B_1\to\partial B'_1$ excursions are the ``big excursions'').
\begin{figure}[t]
\begin{center}
\includegraphics[scale=0.85]{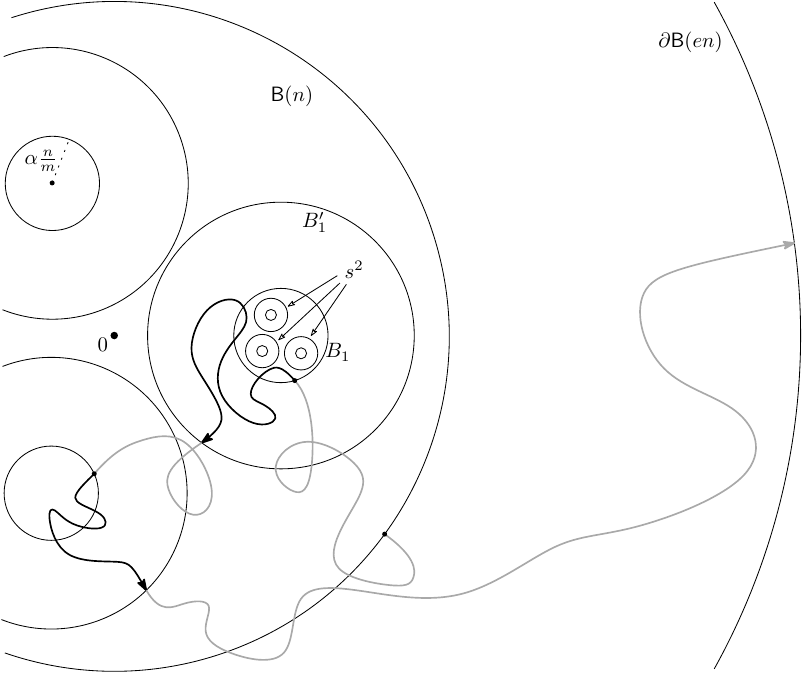}
\caption{On the
proof of Theorem~\ref{t_exc_noncovering}: smaller
excursions generated by a larger one, and nested disks
in two scales.}
\label{f_two_scales}
\end{center}
\end{figure}
Note that, by~\eqref{eq_regularity} (now used with $h=\frac{ms}{\alpha^2}$,
which is at most $\big(\frac{\alpha g_n}{\sqrt{k_n}}\big)^{1+\eps_2}$
when $c\leq 2^{-1-\eps_2}\alpha^2$), we have $k_n-\beta_2g_n \leq k_n-\frac{1}{2}g_n \leq k_{\alpha^2n/ms}$.
Analogously, a smaller pair of nested disks~$(\hB_j,\hB_j')$
is good when
$k_n-\beta_2 g_n$ independent
$\hB_j\to\partial\hB_j'$ excursions
make the event $(E^{\hB_j})^\complement$ occur;
we denote the corresponding event by~$\hG_j$.
Again, it holds that the events
$(\hG_j,j=1,\ldots,s^2)$ are independent
(by the last part of Corollary~\ref{c_SLT_domination}).
Consider the excursions in $(\hB_j,\hB'_j)$
generated by $k_n-\beta_1 g_n$ independent excursions between~$B_1$ and~$\partial B'_1$.
By Corollary~\ref{c_count_SLT}, used with $v=k_n-\beta_1g_n$ and
$t=\min\{(\beta_1-\beta_2)g_n,v\}$ (note that $v+t\leq k_n-\beta_2g_n$,
and that $t\geq C_{10}(\sqrt{v}+v/r)$ with $r=\frac{\alpha^2n}{ms}$
if~$K_0$ is large enough),
with probability at least $1 - C\exp\big(-c\frac{g_n^2}{k_n\ln s}\big)$,
all real excursions in $(\hB_j,\hB'_j)$ are among the
first $\lceil k_n-\beta_2 g_n\rceil$ independent ones of the
free soft local times construction on~$\partial\hB_j$.
Therefore,
the probability that the above holds in all
the smaller nested disks is at least
$1 - Cs^2\exp\big(-c\tfrac{g_n^2}{k_n\ln  s}\big)$.
We also have the following estimate for the
probability that at least one of the smaller pairs of nested disks is good:
\[
\IP\Big[\bigcup_{j=1}^{s^2} \hG_j\Big]
 \geq 1 - 7^{-s^2}.
\]
This implies that
\begin{equation}
\label{est_psi_n}
 \psi_{n,m} \leq 
Cs^2\exp\big(-c\tfrac{g_n^2}{k_n\ln  s}\big) + 7^{-s^2} \leq \exp\big(-c'\tfrac{g_n^2}{k_n\ln(g_n/\sqrt{k_n})}\big).
\end{equation}
Let us mention that
the above is an ``easier'' estimate that would
also work directly for $\IP^{(k_n-g_n)}\big[E^{\B(n)}\big]$;
but we want to make it sharper (basically, by removing the logarithmic factor in the denominator).
This is why we employ this ``two-scale'' construction.

Now, with $a_n=\eps' / \ln m$,
let us define the event
\[
 V_n = \Big\{ m^{-2}\sum_{i=1}^{m^2}\1{G_i}
  \geq 1-a_n\Big\},
\]
i.e., that the proportion of good pairs of nested disks
is at least $1-a_n$.
Then, recall that~$\ZZ_i$ is the set of the $B_i\to\partial B'_i$
excursions generated by the big excursions (now, these are again
the $n\to en$ ones, and there are $k_n-g_n$ of them), and
define the events
\[
 R_i = \big\{\ZZ_i\subset\widetilde\ZZ_i(k_n-\beta_1g_n)\big\}
\]
and
\[
 D_n = \Big\{m^{-2}\sum_{i=1}^{m^2}\1{R_i} > a_n\Big\}
\]
that is, the proportion of nested disk pairs
where the real excursions are among the first
$k_n-\beta_1 g_n$ independent ones is larger
than~$a_n$.
Note that on $V_n\cap D_n$ we have $(E^{\B(n)})^\complement$.

We first use the Chernoff bound
for the Binomial distribution\footnote{If
$X\sim\mathop{Bin}(n,p)$ and $p<a<1$, then
$\IP[X\geq an]\leq \exp\big(-n\HH(a\,\|\,p)\big)$, where $\HH(a\,\|\,p)=a\ln\frac{a}{p}+(1-a)\ln\frac{1-a}{1-p}$. Note that, since $(1-a)\ln\frac{1-p}{1-a}\leq a-p$, we have
$\HH(a\,\|\,p)\geq a\ln\frac{a}{p}-a+p\geq
a\ln\frac{a}{ep}$.}
to obtain that
\begin{equation}
\label{est_V_n}
 \IP[V_n] \geq 1 - \exp\big(- a_n\ln\tfrac{a_n}{e\psi_{n,m}}\times m^2\big)
 \geq 1 - \exp\big(- c' \big(\tfrac{g_n}{\sqrt{k_n}}\big)^{2(1+\eps_2)}
 \big(\ln\tfrac{g_n}{\sqrt{k_n}}\big)^{-2}\big).
\end{equation}

Next, let~$\zeta_i$ be the number of $B_i\to\partial B'_i$ excursions generated
by~$k_n-g_n$ big excursions.
Let us define, with $\beta_1<\beta_4<\beta_3<1$
and~$\Lambda$ to be defined later (right after~\eqref{est_D_n_3})
\begin{align*}
D_n^{(1)} & = \Big\{\sum_{i=1}^{m^2}\zeta_i
\leq m^2(k_n-\beta_3 g_n)\Big\},\\
D_n^{(2)} & = \big\{\zeta_i \geq k_n-\Lambda\text{ for all }i=1,\ldots,m^2\big\},\\
D_n^{(3)} & = \big\{
\ZZ_i(k_n-\beta_4g_n)\subset\widetilde\ZZ_i(k_n-\beta_1 g_n)
 \text{ for all }i=1,\ldots,m^2\big\},
\end{align*}
where~$\Lambda$ is such that $\frac{(\beta_3-\beta_4)g_n}{\Lambda-\beta_4 g_n}> a_n$.
Let
\[
 \NN = \sum_{i=1}^{m^2}\1{\zeta_i\leq k_n-\beta_4 g_n}
\]
be the number of nested disk pairs where
the excursion count is ``not too large''.
If we also have $\zeta_i\geq k_n-\Lambda$ for all~$i$,
this implies on the event~$D_n^{(1)}$ that
\[
 m^2(k_n-\beta_3 g_n) \geq \NN(k_n-\Lambda)
  +(m^2-\NN)(k_n - \beta_4 g_n),
\]
so
\[
 \frac{\NN}{m^2} \geq \frac{(\beta_3-\beta_4)g_n}{\Lambda-\beta_4 g_n}.
\]
Since~$R_i$ occurs when $\zeta_i\leq k_n-\beta_4 g_n$ and~$D_n^{(3)}$
occurs, we obtain that
$D_n\supset D_n^{(1)}\cap D_n^{(2)}\cap D_n^{(3)}$.

Now, Lemma~\ref{l_excurs_control_excursions}
(with $\delta$ of order $g_n/k_n$) implies that
\begin{equation}
\label{est_D_n_1}
 \IP[D_n^{(1)}] \geq
 1-2\exp\Big(-c\frac{g_n^2}{k_n}\Big).
\end{equation}
Then, Lemma~\ref{l_exc_count_one_pair}
together with the union bound
implies that (for $\Lambda\geq 2g_n$; note that~$D_n^{(2)}$
trivially occurs when $\Lambda>k_n$)
\begin{equation}
\label{est_D_n_2}
 \IP[D_n^{(2)}] \geq
 1-2m^2\exp\Big(-c\frac{\Lambda\wedge(\Lambda^2/k_n)}{\ln m}\Big).
\end{equation}
Finally, Corollary~\ref{c_SLT_domination} (used with
$k=\lfloor k_n-\beta_4g_n\rfloor$ and $t=\min\{(\beta_4-\beta_1)g_n,k\}$)
and the union bound imply that
\begin{equation}
\label{est_D_n_3}
 \IP[D_n^{(3)}] \geq
 1-m^2\exp\Big(-c\frac{g_n^2}{k_n}\Big).
\end{equation}
We take $\Lambda=C g_n \ln m$
and recall that $a_n=\eps' / \ln m$.
Then, for large~$C$ and small~$\eps'$, the condition
$\frac{(\beta_3-\beta_4)g_n}{\Lambda-\beta_4 g_n}> a_n$ holds,
and, since $g_n\leq k_n$, the right-hand side of~\eqref{est_D_n_2}
is at least $1-2m^2\exp\big(-c\frac{g_n^2}{k_n}\big)$.
Collecting our estimates
\eqref{est_V_n}--\eqref{est_D_n_3}, we obtain~\eqref{eq_exc_noncovering_upper}.

In the case of decreasing events, the argument is the same
(including the trivial case $g_n<K_0\sqrt{\underline k'_n}$,
where one uses~\eqref{def_hkn_dec}),
with~$\underline k'_n$ in place of~$k_n$
and~\eqref{eq_regularity_dec} in place of~\eqref{eq_regularity}, with the opposite signs
in front of~$g_n$, $\beta_jg_n$, and~$\Lambda$, and with the opposite
inequalities and inclusions for the excursion counts and the sets
of excursions
(so that we use the first inclusions in~\eqref{eq_SLT_domination}
and~\eqref{eq_count_SLT} instead of the second ones).
The only difference is that
$D_n^{(2)}=\big\{\zeta_i \leq \underline k'_n+\Lambda
\text{ for all }i=1,\ldots,m^2\big\}$
is now about the \emph{upper} deviations
of sums of $\underline k'_n+g_n$ terms, and~$\Lambda$ may be
larger than the number of these terms; this is the reason why
Lemma~\ref{l_exc_count_one_pair} is formulated for all $\delta>0$,
and why we take~$\Lambda$ of order $g_n\ln m$
(in the increasing case, $\Lambda$ of order $g_n\sqrt{\ln m}$
would be enough).
This finishes the proof of Theorem~\ref{t_exc_noncovering}.
\end{proof}

Next, we state a technical lemma needed for the proof of the lower bound in Theorem~\ref{t_exc_covering}.
\begin{lem}
\label{l_exc_escape_A}
 Let~$A\subset \B(n)$ be such that $\diam(A)\leq r$
(where $r\geq 1$).
Then, for a universal constant $C\in\R$
(not necessarily positive), we have
\begin{equation}
\label{eq_exc_escape_A}
\IP_{\hm_{\B(n)}^{\B(en)}}[\tau_{\partial\B(en)}<\tau_A]
 \geq 1 - \frac{1}{1\vee(\ln\frac{n}{r} + C)}.
\end{equation}
\end{lem}
\begin{proof}
 Let
$p_A:=
\IP_{\hm_{\B(n)}^{\B(en)}}\big[
\tau_A<\tau_{\partial\B(en)}
\big]$.
Since $p_A\leq 1$, it is enough to prove that
\begin{equation}
\label{eq_p_A_upper}
p_A\Big(\ln\frac{n}{r}+O(1)\Big)\leq 1+O(n^{-1})
\end{equation}
(indeed, since $r\geq 1$, the factor $1+O(n^{-1})$ only changes the
constant in $\ln\frac{n}{r}+O(1)$).
Note that, by Lemma~\ref{l_expected_generated}
and~\eqref{expr_esc_disks} the expected visit
count to~$A$ by one $n\to en$ excursion
(started from the relative harmonic measure~$\hm_{\B(n)}^{\B(en)}$) is
$|A|/\Es_{\B(n)}^{\B(en)}=\frac{2}{\pi}|A|(1+O(n^{-1}))$. On the other hand, standard estimates (e.g., (3.47) of~\cite{P21})
on the Green's function
restricted to~$\B(en)$ imply that, when starting from any
element of~$A$, the expected visit count to~$A$
before hitting~$\partial\B(en)$
is at least $\frac{2}{\pi}|A|(\ln\frac{n}{r}+O(1))$.
We therefore obtain (by conditioning on the first
entrance to~$A$ and using the fact that the last estimate is uniform)
\[
 \frac{2}{\pi}|A|\big(1+O(n^{-1})\big)\geq p_A \times \frac{2}{\pi}|A|\Big(\ln\frac{n}{r}+O(1)\Big),
\]
which implies~\eqref{eq_p_A_upper}.
\end{proof}

\begin{proof}[Proof of Theorem~\ref{t_exc_covering}.]
The proof of the lower bound~\eqref{eq_exc_covering_lower}
is straightforward: we consider first~$k_n$
big excursions, and check if the event $(E^{\B(n)})^\complement$ occurred and if
it happened that~$\vr_n\leq r$ (meaning
that there is some~$A$ with $\diam(A)\leq r$
such that the event $(E^{A})^\complement$ occurred).
If that was the case, we force the remaining~$g_n$
excursions to avoid that~$A$; by Lemma~\ref{l_exc_escape_A},
the conditional probability of that is at least
the second factor in~\eqref{eq_exc_covering_lower}.

Unfortunately, the proof of the lower
bound~\eqref{eq_exc_covering_lower_decreasing}
is not so straightforward in the case of
decreasing events. The reason is that, instead
of adding~$g_n$ independent excursions and paying
for them to avoid the witness set, we would
now have to \emph{remove}~$g_n$ excursions
according to a suitable (possibly random) procedure
and hope that none of these excursions
pass through the witness set;
however, the excursions are no longer independent
 of the witness set (which is random!), so it is not immediately
clear what would be the cost of that.
We therefore need some extra care to deal
with that issue.
Fix~$r$ and $\gamma\geq \gamma_0$. Let~$\mathcal D_{n,r}$ be a deterministic
family of disks of radius~$4r$ such that any $A\subset\B(n)$
with $\diam(A)\leq r$ is contained in one of them, with
$|\mathcal D_{n,r}|\leq c_1(n/r)^2$. For $D\in\mathcal D_{n,r}$,
let~$Q(D)$ denote the number of the first~$\underline{k}'_n$
excursions that visit~$D\cap\B(n)$, and define
\[
H_r(\gamma)=\Big\{
Q(D)\leq \frac{\gamma\underline{k}'_n}{\ln\frac{en}{r}}
\text{ for every }D\in\mathcal D_{n,r}\Big\}.
\]
By Lemma~\ref{l_exc_escape_A} (applied to the sets $D\cap\B(n)$,
whose diameters are at most~$8r$), a binomial large-deviation estimate,
and a union bound over~$\mathcal D_{n,r}$, we have, if~$\gamma_0$ is large enough,
\begin{equation}
\label{eq_H_r_decreasing}
\IP^{(\underline{k}'_n)}\big[H_r(\gamma)^\complement\big]
\leq c_1\Big(\frac{n}{r}\Big)^2
\exp\Big(-c_2\frac{\gamma\underline{k}'_n}
{\ln\frac{en}{r}}\Big).
\end{equation}

On the event~$\{\vr_n\leq r\}$, choose a witness
$A\subset\B(n)$ such that
$\diam(A)\leq r$ and $(E^A)^\complement$ occurs.
The set~$A$ is contained in some
$D\in\mathcal D_{n,r}$, and hence, on $H_r(\gamma)$, at most
\[
q_r(\gamma):=\frac{\gamma\underline{k}'_n}{\ln\frac{en}{r}}
\]
of the excursions visit~$A$.

Now, uniformly choose $\underline{k}'_n-g_n$ of the
$\underline{k}'_n$ excursions. If all excursions visiting~$A$
are retained, then the local-time field on~$A$ is unchanged, and
therefore $(E^A)^\complement$ still occurs. Conditionally on the
$\underline{k}'_n$ excursions, the probability of this event is at least
\[
\frac{\binom{\underline{k}'_n-q_r(\gamma)}{g_n}}
{\binom{\underline{k}'_n}{g_n}}
\geq
\Big(0\vee\Big(1-\frac{c_3\gamma}{\ln\frac{en}{r}}\Big)\Big)^{g_n},
\]
where we used $g_n\leq(1-\eps_1)\underline{k}'_n$ (so that one can take $c_3=\eps_1^{-1}$).
By exchangeability, the retained excursions have the law of
$\underline{k}'_n-g_n$ independent excursions. Consequently,
\[
\IP^{(\underline{k}'_n-g_n)}
\big[(E^{\B(n)})^\complement\big]
\geq
\IP^{(\underline{k}'_n)}
[\vr_n\leq r, H_r(\gamma)] \times
\Big(0\vee\Big(1-\frac{c_3\gamma}{\ln\frac{en}{r}}\Big)\Big)^{g_n}.
\]
Finally,
$\IP^{(\underline{k}'_n)}[\vr_n\leq r,H_r(\gamma)]
\geq
\IP^{(\underline{k}'_n)}[\vr_n\leq r]
-\IP^{(\underline{k}'_n)}[H_r(\gamma)^\complement]$,
and~\eqref{eq_H_r_decreasing} yields~\eqref{eq_exc_covering_lower_decreasing}
(by also taking the supremum over~$r$ and~$\gamma$).

\bigskip

We now prove the upper bound~\eqref{eq_exc_covering}.
Abbreviate $q_n=\IP^{(k'_n+g_n)}\big[(E^{\B(n)})^\complement\big]$.
Fix $m\in \N$ that satisfies~\eqref{df_m_n_covering}; that is,
with $\beta_1:=\frac{1}{2}$, we have
$m\leq \exp\big(\tilde c\frac{g_n}{\sqrt{k'_n}}\big)$ and
$\lfloor k'_n+\beta_1 g_n\rfloor \geq k'_{\alpha^{-1}mn}$
(recall that the thresholds are integers).
Let us consider
$k'_n+\beta_1 g_n$ excursions from
$\B(\alpha^{-1}mn)$ to~$\partial\B(e\alpha^{-1}mn)$.
In $\B(\alpha^{-1}mn)$, consider an admissible collection
$(B_i,B'_i)$, $i=1,\ldots,m^2$ (with $\tilde\alpha=\alpha$), where
$B_i$'s are translates of~$\B(n)$
and $B'_i$'s are translates of~$\B(en)$,
see Figure~\ref{f_bigger_annulus}.
\begin{figure}
\begin{center}
\includegraphics[width=0.85\textwidth]{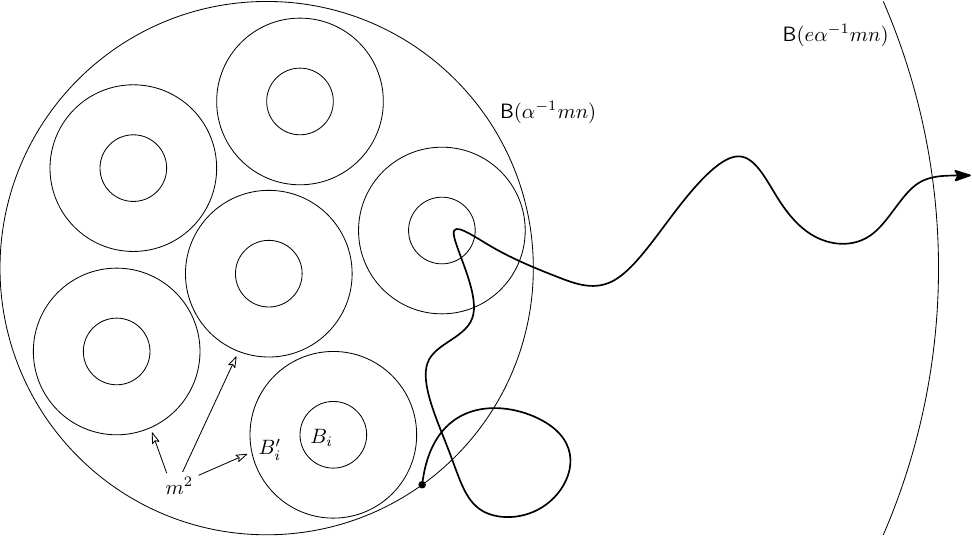}
\caption{Placing $m^2$ copies of $(\B(n),\B(en))$ onto a larger disk.}
\label{f_bigger_annulus}
\end{center}
\end{figure}
For $i=1,\ldots,m^2$,
define the events
\begin{align*}
R_i &=\big\{\ZZ_i\subset\widetilde\ZZ_i(k'_n+g_n)\big\},
\end{align*}
where, as before, $\ZZ_i$ is the set of the $B_i\to\partial B'_i$ excursions
generated by the big excursions (here, the above $k'_n+\beta_1 g_n$ ones);
that is, $R_i$ is the event that in the $i$th translate, real excursions are a subset
of $k'_n+g_n$ independent excursions.
Also, let~$G_i$ be the event that~$E^{B_i}$
occurs with $k'_n+g_n$ independent
$B_i\to\partial B'_i$ excursions (i.e., those
of $\widetilde\ZZ_i(k'_n+g_n)$);
note that $\IP[G_i]=1-q_n$ and
$G_1,\ldots,G_{m^2}$ are independent
(by the last part of Corollary~\ref{c_SLT_domination}).

Let us abbreviate
$\IP_*[\cdot]:=\IP^{(k'_n+\beta_1 g_n)}_{\alpha^{-1}mn}[\cdot]$.
Then, Corollary~\ref{c_count_SLT}, used with $v=k'_n+\beta_1g_n$
and $t=(1-\beta_1)g_n$ (note that we may assume that $m\geq 2$, so that
$g_n\geq\tilde c^{-1}\sqrt{k'_n}\ln 2$ by~\eqref{df_m_n_covering},
and therefore $t\geq C_{10}(\sqrt{v}+v/n)$ if~$\tilde c$ is small enough;
recall also that $g_n\leq k'_n$, so that $t\leq v$,
and that $\ln\frac{em}{\alpha}\leq C\ln m$), implies that
\begin{equation}
\label{est_Uk}
 \IP_*[R_i]\geq 1 -C\exp\Big(-c\frac{g_n^2}{k'_n \ln m}\Big).
\end{equation}
So, using the union bound, we can write
\begin{equation}
\label{est_cup_Ukcompl}
 \IP_*\Big[\bigcup_{i=1}^{m^2} R_i^\complement\Big]
 \leq Cm^2 \exp\Big(-c\frac{g_n^2}{k'_n \ln m}\Big).
\end{equation}

Note that, since $m\leq \exp\big(\tilde c\frac{g_n}{\sqrt{k'_n}}\big)$
(recall also that $m\geq 2$), if~$\tilde c$
is sufficiently small, then
the right-hand side of~\eqref{est_cup_Ukcompl}
does not exceed~$e^{-1}$.
Then, write for large enough~$n$
(note that $E^{B_i}\cap R_i \subset G_i$
and that $\frac{6}{7}>2e^{-1}$)
\begin{align*}
 2e^{-1} &\leq
  \IP_*\big[E^{\B(\alpha^{-1}mn)}\big]\\*
  &\leq
\IP_*\Big[\bigcap_{i=1}^{m^2} E^{B_i}\Big]
  \\*
  &\leq
\IP_*\Big[\bigcup_{i=1}^{m^2} R_i^\complement\Big]
   +
\IP_*\Big[\bigcap_{i=1}^{m^2} (E^{B_i}\cap R_i)\Big]
   \\*
&\leq e^{-1} +
\big(\IP^{(k'_n+g_n)}[E^{\B(n)}]\big)^{m^2}
\\*
  &= e^{-1} + (1-q_n)^{m^2}.
\end{align*}
This means that
$(1-q_n)^{m^2}\geq e^{-1}$,
that is,
\[
 q_n \leq \ln(1-q_n)^{-1}
  \leq m^{-2}.
\]

In the case of decreasing events, the argument is the same,
now with $\underline k_n-\beta_1g_n$ big excursions (the condition
$\underline k_n-\beta_1g_n\leq\underline k_{\alpha^{-1}mn}$ assures that
$\IP_*\big[E^{\B(\alpha^{-1}mn)}\big]\geq\frac{6}{7}$),
and $R_i=\big\{\widetilde\ZZ_i(\underline k_n-g_n)\subset\ZZ_i\big\}$
(we then use the first inclusion in~\eqref{eq_count_SLT},
with $v=\underline k_n-\beta_1g_n$ and $t=(1-\beta_1)g_n$).
This concludes the proof of Theorem~\ref{t_exc_covering}.
\end{proof}

\section{Brownian excursions}
\label{s_brownian}
In this section we argue that Theorems~\ref{t_exc_noncovering}
and~\ref{t_exc_covering} remain valid for independent
\emph{Brownian} excursions. We do not repeat the proofs: we only
describe the setup and indicate the (few) places where
the arguments have to be adapted. Let us mention that the same
approach was adopted in~\cite{CP20}, where a number of results
for the Brownian counterpart of two-dimensional random interlacements
were obtained from the arguments of~\cite{CP17,CPV16} in this way.

In this section, $\B(x,r)=\{y\in\R^2:\|x-y\|\leq r\}$ is the closed
Euclidean disk, $\B^\circ(x,r)$ is its interior,
$\B(r):=\B(0,r)$, and ``disk'' always means a closed disk
of positive radius. For $n\geq 1$, let $W^{(1)},\ldots,W^{(m)}$ be
independent Brownian motions started at uniformly chosen points
of~$\partial\B(n)$ and stopped at the hitting times
$T_1,\ldots,T_m$ of~$\partial\B(en)$ (we still call them
$n\to en$ excursions, and denote by~$\IP_n^{(m)}$
the corresponding probability measure; the uniform measure
on~$\partial\B(n)$ is the counterpart of~$\hm_{\B(n)}^{\B(en)}$).
The field of local times is now replaced by the \emph{occupation measure}
\[
 \mu(G) = \sum_{i=1}^m\int_0^{T_i}\1{W^{(i)}_t\in G\cap\B(n)}\,dt,
 \qquad G\subset\R^2 \text{ Borel}
\]
(in the discrete case, $\ell$ is the density of the analogous measure
with respect to the counting measure), and we write $\mu\leq\mu'$ on~$D$
if $\mu(G)\leq\mu'(G)$ for all Borel $G\subset D$.
For a disk~$D$, let~$\M_f(D)$ be the set of finite Borel measures on~$D$,
endowed with the $\sigma$-algebra generated by the maps $\nu\mapsto\nu(G)$,
$G\subset D$ Borel. An event~$E^D$ is now a measurable subset of~$\M_f(D)$,
and we say that~$E^D$ occurs if $\mu|_D\in E^D$, where~$\mu|_D$ is
the restriction of~$\mu$ to~$D$;
it is increasing if $\mu\in E^D$ and $\mu\leq\mu'$ on~$D$
imply $\mu'\in E^D$ (decreasing and monotone events are then
defined as before).
A translation invariant collection $\EE=(E^D, D\text{ is a disk})$
is called global if~$E^D$ implies~$E^{D'}$ for all disks $D'\subset D$.
Then, the thresholds~\eqref{def_kn}--\eqref{def_hkn_dec},
Condition~L, and the conditions~\eqref{eq_regularity}--\eqref{eq_regularity_dec}
are taken verbatim, and
\[
 \vr_n = 1\vee\inf\big\{\diam(D): D\subset\B(n)\text{ is a disk such that }
  (E^D)^\complement \text{ occurs}\big\}
\]
(with $\inf\emptyset=+\infty$; we assume that~$\vr_n$ is a random variable,
which is clearly the case in the example below).
A natural example is the covering by the Wiener sausage of radius~$1$:
\begin{equation}
\label{df_cov_Brownian}
 E^D=\big\{\mu(\B^\circ(x,1))>0\text{ for all~$x$ such that }\B(x,1)\subset D\big\},
\end{equation}
i.e., all such~$x$ are at distance less than~$1$ from the trajectories;
by scaling, $E^{\B(n)}$ means that $\B(1-n^{-1})$ is covered by
the sausage of radius~$n^{-1}$ of~$m$ independent $1\to e$ excursions.

\begin{theo}
\label{t_brownian}
In the above setup (in particular, assuming that~$\vr_n$ is
a random variable), Theorems~\ref{t_exc_noncovering}
and~\ref{t_exc_covering} hold, \emph{mutatis mutandis}
(and with possibly different constants), for independent
Brownian excursions.
\end{theo}

\begin{proof}[Sketch of the proof.]
The arguments of Section~\ref{s_proofs_main} use the
lattice structure only through the tools of Section~\ref{s_prelims},
the estimates~\eqref{eq_q_half}, and Lemma~\ref{l_exc_escape_A}.
Apart from that, they only rely on the monotonicity and the translation
invariance of the events, on the fact that~$E^D$ implies~$E^{D'}$
for $D'\subset D$, and on the fact that the field restricted to~$D$
only depends on the pieces of the trajectories that are inside~$D$
(this is the reason why we require~$E^D$ to be determined by~$\mu|_D$);
also, all the sets on which the events
are considered there are disks. So, let us go through the above
ingredients.
\begin{itemize}
\item[(i)] Lemmas~\ref{l_interexc_time}, \ref{l_expected_generated},
and~\ref{l_restriction_hm_indep} only rely on the stationarity and
the reversibility of the process on the torus, and therefore remain valid
for the Brownian motion on $\R^2/(N\Z)^2$ (with the Lebesgue measure
in place of the cardinality). For nested disks,
$\Es_{\B(x,r)}^{\B(x,R)}$ is replaced by $\pi\big(\ln\frac{R}{r}\big)^{-1}$
(compare to~\eqref{expr_esc_disks}), and the stationary entrance law
is uniform on~$\partial\B(x,r)$. In particular, one $n\to en$ excursion
generates \emph{exactly} one $\B(x,r)\to\partial\B(x,er)$ excursion
on average.
\item[(ii)] The hitting probabilities for concentric circles are explicit,
so the error terms in~\eqref{eq_excurs_control_Eeta} and in the expression
for~$\IE\zeta_1$ in Lemma~\ref{l_exc_count_one_pair} vanish,
and the proofs of Lemmas~\ref{l_excurs_control_excursions}
and~\ref{l_exc_count_one_pair}
go through without changes.
\item[(iii)] The soft local times method of~\cite{SLT} is formulated
for general state spaces. Here, the Poisson point process on $\partial\A\times\R_+$
has intensity (arc length)$\,\otimes\,du$, and~$\phi(x,\cdot)$ is
the density of the entrance measure with respect to the arc length.
In Lemma~\ref{l_kernel_SLT}, \eqref{hm_vs_hmrel} is now trivial,
\eqref{kernel_Harnack} and~\eqref{kernel_Holder} (with $\beta=1$)
follow by scaling from the smoothness of the density of the exit measure
from an annulus, and we only need~\eqref{kernel_flat}
for $R=er$, where it follows from~\eqref{kernel_Harnack}.
Then, the proofs of Lemma~\ref{l_LD_SLT} and
Corollaries~\ref{c_SLT_domination} and~\ref{c_count_SLT} do not change.
\item[(iv)] In~\eqref{eq_q_half} we now have $q\equiv\frac{1}{2}$
(so one can take $q_\pm=\frac{1}{2}$, and then $V^+_v=V_v=V^-_v$),
the second inequality
there is again a consequence of the Harnack inequality, and
Lemma~\ref{l_exc_escape_A} is obtained in the same way, from the
usual lower bound on the Green's function of the
Brownian motion killed on~$\partial\B(en)$.
\end{itemize}
Also, the infimum in the definition of~$\vr_n$ need not be attained;
so, in the proofs of the lower bounds in Theorem~\ref{t_exc_covering}~(i),
on the event $\{\vr_n\leq r\}$ one takes a witness disk of
diameter at most~$r+\eta$, and then sends~$\eta$ to~$0$. There are no
significant technical difficulties in making this work.
With these modifications, the proofs of Section~\ref{s_proofs_main}
apply.
\end{proof}

As for the applications, consider the covering
event~\eqref{df_cov_Brownian}. As in the discrete case, both~$k_n$
and~$k'_n$ are equal to $2\ln^2 n - (1+o(1))\ln n\ln\ln n$:
this is obtained from the main result of~\cite{BK14}
(see also~\cite{DPRZ04}) in the same way as explained
in Section~\ref{s_covering}
(for the concentration of the excursion counts of the Brownian
motion on the torus, see Section~8 of~\cite{BK14});
also, $(E^{\B(n)})^\complement$ is always witnessed on a disk of
radius~$1$, so $\vr_n=2$ on that event. Therefore,
Corollaries~\ref{c_CT_noncov} and~\ref{c_CT_cov} hold for the Brownian
excursions as well, with the same proofs.
We expect that this is also the case for the results of
Sections~\ref{s_maxdisk} and~\ref{s_LT_extrema} (with, respectively,
the radius of the largest disk covered by the sausage, and the extrema
of $\mu(\B(x,1))$ over the unit disks contained in~$\B(n)$);
however, this would require the Brownian analogues
of Theorem~\ref{t_max_disk_typical}
and of~\eqref{eq_extrema_typical},
and we do not elaborate on that here.

\section*{Appendix: the typical size of the maximal covered disk}
\begin{proof}[Proof of Theorem~\ref{t_max_disk_typical}.]
This result can be obtained by adapting the multiscale arguments
of~\cite{A15,DPRZ06} (see also the other papers mentioned before
Theorem~\ref{t_max_disk_typical}) as described below, and we do not
claim it as new; therefore, we only sketch its proof, referring to
these papers whenever possible.
For $b\in(0,1)$ and~$x$ such that $\B(x,en^b)\subset\B(n)$,
let~$N_x$ be the number
of $\B(x,n^b)\to\partial\B(x,en^b)$ excursions generated by
the $2\theta\ln^2 n$ big (i.e., $n\to en$) excursions.
The proof is based on the following two facts:
one needs $2b^2\ln^2 n\,(1+o(1))$ such excursions to
cover~$\B(x,n^b)$, and, for $a>\sqrt{\theta}$,
\begin{equation}
\label{eq_LD_exc_count}
 \IP\big[N_x\geq 2a^2\ln^2 n\big]
  \leq n^{-\frac{2(a-\sqrt{\theta})^2}{1-b}+o(1)}
\end{equation}
(one can show that the exponent here is sharp, but we only need
the bound as stated).
Since there are $n^{2(1-b)}$ ``essentially different''
disks of radius~$n^b$, the largest~$a$ for which one can find~$x$
with $N_x\geq 2a^2\ln^2 n$ should then be
$\sqrt{\theta}+1-b$, and we have $\sqrt{\theta}+1-b> b$
if and only if $b<\chi(\theta)$.

Let us first explain why~\eqref{eq_LD_exc_count} holds. By Lemma~\ref{l_exc_escape_A}
and Lemma~3.11 of~\cite{P21} (as in the proof of
Lemma~\ref{l_exc_count_one_pair}), the number of
$\B(x,n^b)\to\partial\B(x,en^b)$ excursions generated by one
big excursion is dominated by a random variable~$\xi$ which
equals~$0$ with probability $1-p$ and is Geometric($h$)
with probability~$p$, where
$p,h=\frac{1+o(1)}{(1-b)\ln n}$. Then, \eqref{eq_LD_exc_count} follows from the exponential
Chebyshev inequality
(used with $\lambda=\big(1-\frac{\sqrt{\theta}}{a}\big)h$).

Now, to obtain the upper bound in~\eqref{eq_max_disk_typical},
let $b=\chi(\theta)+\frac{\eps}{2}$ and take
$a\in(\sqrt{\theta}+1-b,b)$. By
Corollary~\ref{c_SLT_domination} and
the upper bound of Theorem~\ref{t_exc_noncovering}
(used for the covering event, as in Corollary~\ref{c_CT_noncov}),
we have
$\IP\big[\B(x,n^b)\text{ is covered},\,N_x<2a^2\ln^2 n\big]\leq e^{-c\ln^2 n}$.
Together with~\eqref{eq_LD_exc_count}, this shows that
the probability that~$\B(x,n^b)$ is covered is at most
$n^{-2(1-b)-\delta}$ for some $\delta>0$,
and we conclude with the union bound over
$O(n^{2(1-b)})$ disks (note that any disk of radius~$4n^b$
contains $\B(x,en^b)$ for some $x\in n^b\Z^2$); so, with high probability,
$\hR^{\B(n)}<4n^b\leq n^{\chi(\theta)+\eps}$.

As for the lower bound, we use
the second-moment argument of Section~3 of~\cite{A15}
(which goes back to~\cite{DPRZ06}); since all the necessary estimates
are contained in these two papers, we only indicate how they are used.
Let $b=\chi(\theta)-\eps$ and take $a\in(b,\sqrt{\theta}+1-b)$
(note that $b>\sqrt{\theta}$ if~$\eps$ is small enough).
Consider the SRW~$X$ on the torus~$\Z^2_K$ with $K=\lceil 9n\rceil$,
started outside~$\B(en)$.
In the rest of the proof we use the notation of~\cite{A15},
with~$L$ in place of~$n$ there (in particular,
the disks~$D(x,r)$ and their boundaries are those of~\cite{A15}):
let~$L$ and the radii $r_j=e^{L}L^{3(L-j)}$, $0\leq j\leq L$, be
as in~(3.5) there, so that $K=L^{\bga}r_0$ with $\bga\in[\hat b,\hat b+4]$,
where~$\hat b$ is a large constant.
Note that $\ln n=(3+o(1))L\ln L$ and $r_0=n\ln^{-\bga+o(1)}n$:
for the second-moment estimate, it is essential that
the largest scale is separated from the size of the torus
by a large power of the logarithm.
Let $d=\frac{a-\sqrt{\theta}}{1-b}\in(0,1)$ and
\[
 u_j=\Big\lceil 6\big(1-L^{-1/4}\big)\Big(\sqrt{\theta}+d\,\frac{j}{L}\Big)^2
 L^2\ln L\Big\rceil, \qquad 0\leq j\leq L-1
\]
(these are the quantities~$n_j$ of~\cite{A15}, with~$\eta$ given by
$\sqrt{\theta+2\eta\sqrt{\theta}}-\sqrt{\theta}=d$; the assumption
$\eta<1+\frac{1}{2\sqrt{\theta}}$ there means that $d<1$,
i.e., $a<\sqrt{\theta}+1-b$),
and let~$J$ be the largest integer such that $r_{J+1}>n^b$.
Then $J=(1-b+o(1))L$ and $u_J=(6+o(1))a^2L^2\ln L$; since the ratio
of the radii is now~$L^3$ instead of~$e$, the latter corresponds to
$N_x=2a^2\ln^2 n$.
We say that $x\in\B(n/2)$ is \emph{successful} if
$|\hN^x_j-u_j|\leq L$ for all $1\leq j\leq J$, where~$\hN^x_j$ is
the number of excursions from~$\partial D(x,r_{j+1})$
to~$\partial D(x,r_j)$ up to the moment~$\tau^x$ when~$X$ completes
its first~$u_0$ excursions from~$\partial D(x,r_1)$
to~$\partial D(x,r_0)$ and back.
This is Definition~3.1 of~\cite{A15}, with the only difference
that the scales below~$r_{J+1}$ are not used
(as in~(4.2) of~\cite{DPRZ06}).
Let~$Z$ be the number of successful sites in~$\B(n/2)$.
The proofs of Propositions~3.4 and~3.5 of~\cite{A15} treat the scales
one by one (and only concern the sequence of sites visited by the walk,
so it does not matter that the walk of~\cite{A15} is in continuous time);
it can then be obtained from the arguments there that
$\IP[x\text{ is successful}]=(1+o(1))\hat q$ uniformly in~$x$,
where~$\hat q$ does not depend on~$x$ and satisfies
$\hat q=n^{-\frac{2(a-\sqrt{\theta})^2}{1-b}+o(1)}$, and that,
if~$\hat b$ is large enough,
\[
 \IE Z^2\leq(1+o(1))(\IE Z)^2
\]
(it is here that one needs $a<\sqrt{\theta}+1-b$).
Note that, after the truncation, the analogue of Proposition~3.5~(i)
of~\cite{A15} only covers the pairs of sites with $\ell(x,y)\leq J-1$
(in the notation there); the remaining pairs are
at distance at most $2r_{J-1}+2=n^{b+o(1)}$ from each other, and
bounding the probability that both sites are successful
by~$(1+o(1))\hat q$ shows that their contribution to~$\IE Z^2$ is at most
$n^{2+2b+o(1)}\hat q=o\big((\IE Z)^2\big)$,
again because $a-\sqrt{\theta}<1-b$.

Next, let $Z'\leq Z$ be the number of successful sites $x\in\B(n/2)$
such that, in addition, $D(x,r_{J+1})\supset\B(x,n^b)$ is covered
by time~$\tau^x$. Since $a>b$, we have $\IE Z'=(1-o(1))\IE Z$.
Indeed, let~$\G^x$ be the $\sigma$-algebra generated by the excursions
of~$X$ from~$\partial D(x,r_{J-1})$ to~$\partial D(x,r_J)$, as in
Lemma~2.4 of~\cite{DPRZ06} (used with $r=r_{J+1}$, $R=r_J$, $R'=r_{J-1}$).
The events $\{|\hN^x_j-u_j|\leq L\}$, $j\leq J-1$, belong to~$\G^x$,
and, on $\{\hN^x_{J-1}=m\}$, both~$\hN^x_J$ and the trace of~$X$
on~$D(x,r_{J+1})$ up to time~$\tau^x$ are determined by the
excursions from~$\partial D(x,r_{J+1})$ to~$\partial D(x,r_J)$ that occur
during the first~$m$ excursions from~$\partial D(x,r_J)$
to~$\partial D(x,r_{J-1})$. By that lemma, uniformly in~$m$ such that
$|m-u_{J-1}|\leq L$, the conditional (on~$\G^x$) law of the former
excursions differs from the unconditional one by a factor
$1+O\big(m\frac{r_{J+1}}{r_J}\ln\frac{r_J}{r_{J+1}}\big)=1+o(1)$.
Under the latter, their number belongs
to $[u_J-L,u_J+L]$ with
probability at least~$L^{-C}$ (cf.~(3.15) of~\cite{A15}); on the other hand,
it is elementary to obtain that each such excursion
visits a given site of~$D(x,r_{J+1})$ with probability at least
$\frac{(3+o(1))\ln L}{b\ln n}$ uniformly in its starting point,
so the probability that some site of~$D(x,r_{J+1})$ is not visited
by the first $u_J-L$ of them is at most
$n^{2b-\frac{2a^2}{b}+o(1)}=o(L^{-C})$.
Then, by the Cauchy--Schwarz inequality (and since $0\leq Z'\leq Z$),
$\IP[Z'>0]\geq(\IE Z')^2/\IE Z'^2\geq(\IE Z')^2/\IE Z^2\to 1$.

Finally, we go back to independent excursions.
By~(3.19) of~\cite{DPRZ06} (used with $r=r_1$, $R=r_0$, and the relative
deviation~$L^{-1/4}$; recall the factor $1-L^{-1/4}$ in~$u_0$, and note
that $\ln K\geq 3L\ln L$), for each~$x$ we have
$\tau^x\leq t^*:=\frac{4\theta}{\pi}K^2\ln^2K$ with probability at
least $1-\exp(-cL^{3/2}\ln L)$; since $\ln n=(3+o(1))L\ln L$,
by the union bound, with high probability this holds for all $x\in\B(n/2)$.
Fix $\eps'>0$; by Lemma~2.11 of~\cite{CP17},
with high probability at most $2(\theta+\eps')\ln^2 n$ excursions of~$X$
between~$\partial\B(n)$ and~$\partial\B(en)$ start before time~$t^*$,
and, by Corollary~\ref{c_SLT_domination}, with high probability these
are among the first $2(\theta+2\eps')\ln^2 n$ independent ones.
So, with high probability, $2(\theta+2\eps')\ln^2 n$ independent
excursions cover a disk of radius~$n^{\chi(\theta)-\eps}$ in~$\B(n)$.
Applying this with some $\theta_0<\theta$ in place of~$\theta$
and with~$\eps'$ such that $\theta_0+2\eps'<\theta$, and using that
$\{\hR^{\B(n)}\geq r\}$ is an increasing event, we obtain that
$\IP^{(2\theta\ln^2 n)}\big[\hR^{\B(n)}\geq n^{\chi(\theta_0)-\eps}\big]\to 1$;
since~$\chi$ is continuous and~$\eps$ is arbitrary, this implies
the lower bound in~\eqref{eq_max_disk_typical}.
\end{proof}

\section*{Acknowledgements}
The authors thank Nina Gantert for many 
useful discussions on this subject.
O.C.\ and S.P.\ were partially supported by 
CMUP, member of LASI, which is financed by national funds
through FCT --- Funda\c{c}\~ao
para a Ci\^encia e a Tecnologia, I.P.,
under the project with reference UID/00144/2025,
\url{https://doi.org/10.54499/UID/00144/2025}.
This research was funded in part by the Austrian Science Fund (FWF) 10.55776/F1002. 

During the final stage of this work, large language models 
were used for checking 
the validity of the arguments, for writing 
some technical parts in a detailed way, and
generally as writing assistants; however, 
the high-level mathematical ideas
and proof strategies are entirely human-generated.
The authors take full responsibility for the content of the paper.

For open access purposes, the authors have applied a CC BY public copyright license to any author accepted manuscript version arising from this submission.


\begin{thebibliography}{23}
\bibitem{A15}
\textsc{Y.~Abe} (2015)
Maximum and minimum of local times for two-dimensional random walk.
\textit{Electron.\ Commun.\ Probab.} \textbf{20}, 1--14.

\bibitem{A21}
\textsc{Y.~Abe} (2021)
Second-order term of cover time for planar simple random
walk.
\textit{J.\ Theoret.\ Probab.} \textbf{34} (3), 1689--1747.


\bibitem{BK14}
\textsc{D.~Belius, N.~Kistler} (2017)
The subleading order of two dimensional cover times.
\textit{Probab.\ Theory Relat.\ Fields} \textbf{167} (1),
461--552.

\bibitem{BRZ20}
\textsc{D.~Belius, J.~Rosen, O.~Zeitouni} (2020)
Tightness for the cover time of the two dimensional sphere.
\textit{Probab.\ Theory Relat.\ Fields}
\textbf{176}, 1357--1437.

\bibitem{BGP18}
\textsc{D.F.~de Bernardini, C.F.~Gallesco, S.~Popov} (2018)
On uniform closeness of local times of
Markov chains and i.i.d.\ sequences.
\textit{Stochastic Process.\ Appl.} \textbf{128} (10),
3221--3252.

\bibitem{B02}
\textsc{O.~Bousquet} (2002)
A Bennett concentration inequality and its application to suprema
of empirical processes.
\textit{C.\ R.\ Math.\ Acad.\ Sci.\ Paris} \textbf{334} (6), 495--500.


\bibitem{CT14} \textsc{J.~\v{C}ern\'y, A.~Teixeira} (2016)
Random walks on torus and random interlacements:
Macroscopic coupling and phase transition.
\textit{Ann.\ Appl.\ Probab.}, \textbf{26} (5), 2883--2914.

\bibitem{CP24}
\textsc{O.~Collin, S.~Popov} (2024)
Two-dimensional random interlacements: 0-1 law
and the vacant set at criticality.
\textit{Stochastic Process.\ Appl.}
\textbf{169}, 104272.

\bibitem{CGPV13} \textsc{F.~Comets, C.~Gallesco, S.~Popov,
M.~Vachkovskaia} (2013)
On large deviations for the cover time of
two-dimensional torus.
\textit{Electron.\ J.\ Probab.}, \textbf{18}, article 96.

\bibitem{CP17} \textsc{F.~Comets, S.~Popov} (2017)
 The vacant set of two-dimensional critical
  random interlacement is infinite.
 \textit{Ann.\ Probab.} \textbf{45} (6B), 4752--4785.

\bibitem{CP20} \textsc{F.~Comets, S.~Popov} (2020)
  Two-dimensional Brownian random interlacements.
\textit{Potential Analysis}, \textbf{53} (2), 727--771.

\bibitem{CPV16} \textsc{F.~Comets, S.~Popov, M.~Vachkovskaia} (2016)
Two-dimensional random interlacements and late points
for random walks.
\textit{Commun.\ Math.\ Phys.} \textbf{343}, 129--164.

\bibitem{DPR07}
\textsc{A.~Dembo, Y.~Peres, J.~Rosen} (2007)
How large a disc is covered by a random walk in~$n$ steps?
\textit{Ann.\ Probab.} \textbf{35} (2), 577--601.

\bibitem{DPRZ04}
\textsc{A.~Dembo, Y.~Peres, J.~Rosen, O.~Zeitouni} (2004)
Cover times for Brownian motion and random walks in two dimensions.
\textit{Ann.\ Math.\ (2)} \textbf{160} (2), 433--464.

\bibitem{DPRZ06}
\textsc{A.~Dembo, Y.~Peres, J.~Rosen,
O.~Zeitouni}
(2006)
Late points for random walks in two dimensions.
\textit{Ann.\ Probab.} \textbf{34} (1), 219--263.

\bibitem{DRZ26}
\textsc{A.~Dembo, J.~Rosen, O.~Zeitouni} (2026)
Limit law for Brownian cover time of the two dimensional torus.
arXiv:2610.03555

\bibitem{D14}
\textsc{J.~Ding} (2014)
Asymptotics of cover times via Gaussian free fields:
Bounded-degree graphs and general trees.
\textit{Ann.\ Probab.} \textbf{42} (2), 464--496.


\bibitem{LL10}
\textsc{G.~Lawler, V.~Limic} (2010)
\textit{Random Walk: A Modern Introduction.}
Cambridge Studies in Advanced Mathematics, \textbf{123}.
Cambridge University Press, Cambridge.

\bibitem{LSX26}
\textsc{X.~Li, J.~Shi, Q.~Xu} (2026)
Large deviations of cover time of tori in dimensions $d \geq 3$.
To appear in:
\textit{Ann.\ Inst.\ H.\ Poincar\'e\ Probab.\ Stat.}
arXiv:2411.16398







\bibitem{P21}
\textsc{S.~Popov} (2021)
\textit{Two-dimensional Random Walk:
From Path Counting to Random Interlacements.}
Cambridge University Press, Cambridge.

\bibitem{SLT}
\textsc{S.~Popov, A.~Teixeira} (2015)
Soft local times and decoupling of random interlacements.
\textit{J.\ European Math.\ Soc.} \textbf{17} (10), 2545--2593.



\bibitem{vdVW11}
\textsc{A.~van der Vaart, J.A.~Wellner} (2011)
A local maximal inequality under uniform entropy.
\textit{Electron.\ J.\ Statist.} \textbf{5}, 192--203.

\bibitem{Z18}
\textsc{A.~Zhai} (2018)
Exponential concentration of cover times.
\textit{Electron.\ J.\ Probab.} \textbf{23}, 1--22.

\end{thebibliography}
\end{document}